\documentclass[12pt,letterpaper,twoside,english]{amsart}
\usepackage[T1]{fontenc}
\usepackage[latin9]{inputenc}
\usepackage{mathrsfs}
\usepackage{comment}
\usepackage{amsthm}
\usepackage{amstext}
\usepackage{amssymb}
\usepackage{tikz}
\usepackage[bbgreekl]{mathbbol}
\usepackage{amsfonts}
\DeclareSymbolFontAlphabet{\mathbb}{AMSb}
\DeclareSymbolFontAlphabet{\mathbbl}{bbold}
\DeclareFontFamily{U}{wncyss}{}
\DeclareFontShape{U}{wncyss}{m}{n}{<->wncyss8}{}
\DeclareSymbolFont{mcy}{U}{wncyss}{m}{n}
\DeclareMathSymbol{\ra}{\mathord}{mcy}{"61}
\DeclareMathSymbol{\rb}{\mathord}{mcy}{"62}
\DeclareMathSymbol{\rd}{\mathord}{mcy}{"64}
\DeclareMathSymbol{\rg}{\mathord}{mcy}{"67}
\DeclareMathSymbol{\rka}{\mathord}{mcy}{"6B}
\usepackage{graphicx}
\usepackage{setspace}
\usepackage{esint}
\usepackage{enumerate}
\usepackage{amscd,color}

\makeatletter

\pdfpageheight\paperheight
\pdfpagewidth\paperwidth

\numberwithin{equation}{section}
\numberwithin{figure}{section}
\theoremstyle{plain}
\newtheorem{thm}{\protect\theoremname}[section]
  \theoremstyle{definition}
  \newtheorem{example}[thm]{\protect\examplename}
  \theoremstyle{remark}
  \newtheorem{rem}[thm]{\protect\remarkname}
  \theoremstyle{plain}
  \newtheorem{prop}[thm]{\protect\propositionname}
  \theoremstyle{plain}
  \newtheorem{cor}[thm]{\protect\corollaryname}
  \theoremstyle{definition}
  \newtheorem{defn}[thm]{\protect\definitionname}
  \theoremstyle{plain}
  \newtheorem{lem}[thm]{\protect\lemmaname}
 \theoremstyle{remark}

\theoremstyle{plain}

\theoremstyle{plain}

\usepackage{enumitem}
\usepackage{amsmath}
\usepackage{amsfonts}
\usepackage{amssymb}
\usepackage[all]{xy}
\usepackage{array}

\def\RR{\mathbb{R}}
\def\CC{\mathbb{C}}
\def\QQ{\mathbb{Q}}
\def\PP{\mathbb{P}}
\def\ZZ{\mathbb{Z}}
\def\HH{\mathbb{H}}
\def\VV{\mathbb {V}}
\def\str{\mathsf{S}}

\def\ff{\mathsf{f}}
\def\fg{\mathsf{g}}

\def\fx{\mathfrak{X}}
\def\cE{\mathcal{E}}

\def\cx{\mathcal{X}}
\def\cs{\mathcal{S}}
\def\ch{\mathcal{H}}
\def\pch{{}^p \ch}

\def\cy{\mathcal{Y}}
\def\cl{\mathcal{L}}

\def\co{\mathcal{O}}
\def\cv{\mathcal{V}}

\def\ppsi{{}^p \psi}
\def\pphi{{}^p \phi}

\def\gr{\mathrm{Gr}}

\def\sp{\mathtt{sp}}

\def\pha{\mathrm{ph}}
\def\van{\mathrm{van}}
\def\lm{\mathrm{lim}}

\def\IH{\mathrm{IH}}
\def\IC{\mathrm{IC}^{\bullet}}

\def\eb{\epsilon^{\bullet}}

\def\uo{\underline{0}}

\def\ux{\underline{x}}

\def\sfs{\mathsf{S}}

\def\UZ{\underline{Z}}
\def\mo{\vec{\mathsf{p}}}
\def\ms{\mathsf{m}}
\def\N{\mathfrak{N}}
\def\M{\mathfrak{M}}
\def\a{\alpha}
\def\disj{\sqcup}
\def\stm{\setminus}
\def\I{\mathscr{I}}
\def\ev{\mathrm{ev}}
\def\TP{\tilde{\PP}^6}
\def\TX{\tilde{X}_0}
\def\b{\mathsf{b}}
\def\B{\mathsf{B}}
\def\EE{\mathbb{E}}
\def\EB{\EE_{\B}}
\def\eb{\EE_{\b}}

\let\oldtocsection=\tocsection
 
\let\oldtocsubsection=\tocsubsection
 
\let\oldtocsubsubsection=\tocsubsubsection
 
\renewcommand{\tocsection}[2]{\hspace{0em}\oldtocsection{#1}{#2}}
\renewcommand{\tocsubsection}[2]{\hspace{1em}\oldtocsubsection{#1}{#2}}
\renewcommand{\tocsubsubsection}[2]{\hspace{2em}\oldtocsubsubsection{#1}{#2}}

\makeatother

  \providecommand{\corollaryname}{Corollary}
  \providecommand{\definitionname}{Definition}
  \providecommand{\examplename}{Example}
  \providecommand{\lemmaname}{Lemma}
  \providecommand{\propositionname}{Proposition}
  \providecommand{\remarkname}{Remark}
\providecommand{\theoremname}{Theorem}

\providecommand{\conventionname}{Convention}

\begin{document}

\title[Hodge theory of degenerations, III]{Hodge theory of degenerations, (III): \\ a vanishing-cycle calculus for non-isolated singularities}

\author{Matt Kerr}
\address{Washington University in St. Louis, Department of Mathematics and Statistics, St. Louis, MO 63130-4899}
\email{matkerr@math.wustl.edu}

 \author{Radu Laza}
 \address{Stony Brook University, Department of Mathematics, Stony Brook, NY 11794-3651}
\email{rlaza@math.stonybrook.edu}

\thanks{The authors are supported in part by NSF grant DMS-2101482 (MK) and DMS-2101640 (RL)}

\bibliographystyle{amsalpha}


\begin{abstract}
We continue our study of the Hodge theory of degenerations, initiated in \cite{KL1} (Part I - general theory) and \cite{KL} (Part II - geometric applications in the isolated singularity case). The focus here is on concrete computations in the case of non-isolated singularities, particularly those for which the singular locus has dimension one.  These examples are significantly more involved than in the previous parts, and include $k$-log-canonical singularities, several specific surface singularities (both slc and non-slc), and certain singular 5-folds arising in the study of Feynman integrals.
\end{abstract}

\maketitle

\tableofcontents

\section*{Introduction}
In this third and final chapter of our study of the interplay between singularities of the central fiber $X_0$ of a degeneration $\cx/\Delta$ and its limiting mixed Hodge structure (\cite{KL1, KL}; see also \cite{KLS}), we turn our focus to the setting where $X_0$ has non-isolated singularities.  The situation here is more subtle and far less chronicled than the isolated singularity case, which is already the subject of a vast literature (cf.~\cite{KL} and references within).  At the core of this installment is a series of examples (discussed in \S\S\ref{S7.3}--\ref{S7.6}) originating in the study of compactifications of moduli spaces of surfaces as well as Feynman amplitudes in mathematical physics.  We believe that our analysis of these examples will prove useful both in relation to the concrete geometric contexts from which they are drawn, and as an illustration of a general calculus in the non-isolated case.

Throughout this series of papers, we have considered the setting of a projective morphism $f\colon \mathcal{X}\to \Delta$ from an irreducible complex analytic space of dimension $n+1$ to the disk, which extends to a projective morphism of quasi-projective varieties.  Here we assume in addition that the $X_t:=f^{-1}(X_t)$ are smooth for $t\neq 0$; and in \S\S\ref{S7.1}-\ref{S7.4} below that $\mathcal{X}$ is also smooth, so that $X_0=f^{-1}(0)$ has hypersurface singularities.  The theme as always is to relate the singularity type to the limiting MHS of the degeneration and the MHS on the singular fiber, in large part through a detailed analysis of the vanishing cohomology.

Recall that the \emph{vanishing-cycle sequence}
\begin{equation}\label{eq7.0}
\to H^k(X_0)\overset{\sp}{\to} H^k_{\lm}(X_t)\to H^k_{\van}(X_t)\overset{\delta}{\to} H^{k+1}(X_0)\overset{\sp}{\to}H^{k+1}_{\lm}(X_t)\to.
\end{equation}
is an exact sequence of MHSs compatible with the action of $T^{\text{ss}}$, the semisimple part of monodromy about $t=0$.  In the present installment, our main purpose is to revisit and extend a formula for the spectrum of a non-isolated hypersurface singularity, and apply it to various examples arising in GIT, MMP, and theoretical physics.  This ``SSS formula'' concerns the case of $\dim(\text{sing}(X_0))=1$, where (by Theorem \textbf{I.5.5(ii)})\footnote{Since they are frequent, references to \cite{KL1} and \cite{KL} will be done in the format \textbf{I.*} and \textbf{II.*}, where \textbf{*} is the section, theorem or equation number.} $H^k_{\van}(X_t)$ is potentially nonzero only for $k=n-1,n,n+1$.   The formula was conjectured by Steenbrink \cite{St89}, then proved by Siersma \cite{Si90} modulo $\ZZ$ and by M. Saito \cite{Sa-steen} in full.  Here we explain how to extend it to weighted spectra, see Theorem \ref{th-sss} and Example \ref{ex-key}. It allows us to compute the mixed Hodge module $\pphi_f \QQ_{\cx}[n+1]$ supported on $\text{sing}(X_0)$, hence (via the spectral sequence \eqref{eq-SS*}) the MHSs and $T^{\text{ss}}$-actions on the $H^k_{\van}(X_t)$ in \eqref{eq7.0}.
 
We carry out this computation, in particular, for all the semi-log-canonical hypersurface singularities in dimension $n=2$ (classified in \cite{LR}) as well as the class $J_{k,\infty}$ of non-slc surface singularities appearing for example in the study of compactified moduli spaces of $K3$'s \cite{log2}.  See \S\S\ref{S7.3}-\ref{S7.4}. In particular, this leads to a lower bound of $\lfloor\tfrac{k-1}{2}\rfloor$ on the geometric genus of $X_t$ when $X_0$ has a $J_{k,\infty}$ singularity, cf.~Theorem \ref{th7.4}.

For applications to 1-parameter degenerations of hypersurfaces in projective space, it is unrealistic to expect a smooth total space $\mathcal{X}$:  the base-locus will meet $Z=\text{sing}(X_0)$.  However, if a Zariski open $Z_0\subset Z$ has type $A_{\infty}$ singularities, we can arrange for this intersection to occur in $Z_0$ (with all multiplicities $1$), making the singularities of $\mathcal{X}$ only nodal.  While \eqref{eq7.0} still holds, this disrupts the Clemens-Schmid sequence (including the local-invariant cycle property), as well as the computation of the vanishing cycle complex.  In \S\ref{S7.5}, we quantify both disruptions in general and give applications to degenerations of surfaces; in particular, the assumptions of Theorem \ref{th7.4} can be weakened to allow for a nodal total space $\mathcal{X}$ (Corollary \ref{cor-dis}). 

The final section \S\ref{S7.6} is a detailed example involving cubic 5-folds defined by the second Symanzik polynomial arising from the ``double-box'' Feynman diagram studied in \cite{Bl21,DHV23}.  (For a discussion of their origin in quantum field theory and a brief history of the mathematical investigation of Feynman graph hypersurfaces more generally, see the Introduction to \cite{DHV23}.)  Our approach here is to consider a smoothing with nodal total space and apply the SSS formula together with the tools of \S\ref{S7.3} and \S\ref{S7.5} and some detailed sheaf-cohomology arguments.  This allows us to compute all the Hodge numbers of the double-box 5-fold $X_0$ and the monodromy of the degeneration in full, deepening the results in [op.~cit.].  We briefly address the geometric meaning of some extension classes appearing in $H^5(X_0)$.

To set the stage for these more specialized results, we begin in \S\ref{S7.1} with a few general ones with no constraint on $\dim(\text{sing}(X_0))$, which make no use of SSS.  A key result here is that the MHSs on the (local) Milnor fibers and (global) vanishing cohomology of a degeneration with $k$-log-canonical singularities are contained in the $(k+1)^{\text{st}}$ step of the Hodge filtration (Theorem \ref{th-klc}).  (It turns out that the double-box 5-folds have 1-log-canonical --- in fact, 1-rational --- singularities.)  We also give a short proof of a result of \cite{EFM} on torsion exponents for the base-change of a semi-stable degeneration.

Because they are used systematically in the paper, we remind the reader of how (weighted) spectra are defined.  Write $\mathbf{e}(\alpha):=e^{2\pi\mathbf{i}\alpha}$.  
\begin{defn}\label{def-spectrum}
Given a MHS $V$ with a finite automorphism $\gamma$, let $V_{\CC}=\oplus V^{p,q}$ be the Deligne bigrading and $E_{\lambda}(-)$ denote eigenspaces of $\gamma$.  The corresponding \emph{spectrum} $\sigma(V,\gamma)=\sum_{\alpha \in \QQ} m_{\alpha} [\alpha]\in \ZZ[\QQ]$ (free abelian group) is given by $m_{\alpha}:=\dim(E_{\mathbf{e}(\alpha)}(\gr_F^{\lfloor\alpha\rfloor}V_{\CC}))$.  The \emph{weighted spectrum} $\tilde{\sigma}(V,\gamma)=\sum_{\alpha,w}m_{\alpha,w}[(\alpha,w)]\in \ZZ[\QQ\times \ZZ]$ is given by $m_{\alpha,w}:=\dim(E_{\mathbf{e}(\alpha)}(V^{\lfloor\alpha\rfloor,w-\lfloor\alpha\rfloor}))$.
\end{defn}
\noindent Typically in this paper, $V$ is the reduced cohomology of a Milnor fiber in a degeneration as above, $\tilde{H}^k(\imath_p^*\phi_f \QQ_{\mathcal{X}})\cong \tilde{H}^k(\mathfrak{F}_{f,p})$ ($p\in \text{sing}(X_0)$), and $\gamma$ is the semisimple part $T^{\text{ss}}$ of the monodromy operator.  In this case the spectra are denoted $\sigma^k_{f,p}$ resp.~$\tilde{\sigma}^k_{f,p}$.

\subsection*{Acknowledgments}
The authors thank the IAS for providing the environment in which, some years ago, this series of papers was first conceived.  
We also thank M.~Saito, A.~Harder and B.~Castor for helpful remarks, discussion, and correspondence.

\section{Remarks on various singularity classes}\label{S7.1}

One theme of our work has been to identify conditions on the singularities of the central fiber $X_0$ of a degeneration under which the limit mixed Hodge structure is determined to some extent by the mixed Hodge structure on $H^*(X_0)$. The quintessential result here is the Clemens-Schmid exact sequence \cite{Clemens}, especially in the case of a semistable degeneration.  While it is valid more generally for $\cx$ smooth (cf.~for instance \cite{KL1}), in the semistable case unipotency of the monodromy operator $T$ means that more of $H^*_{\lim}(X_t)$ is invariant hence comes from $H^*(X_0)$, which in turn is often easier to compute as $X_0$ is a normal-crossing variety.

However, in higher dimensions it is preferable not to modify the central fiber $X_0$ and to allow more complicated singularities. At the same time, we would like to know that some of $H^*_{\lim}$ is constrained by $H^*(X_0)$. In this direction, Steenbrink and Du Bois in the 80's showed that the MHS on an $X_0$ with Du Bois (resp.~rational) singularities controls the $\gr_F^0$ (resp.~ the ``frontier'') of the limiting MHS; \cite{KLS} contains a fairly optimal result in this direction, which is restated below in Theorem \ref{t7.1}. Recently, the notions of higher Du Bois (or, equivalently in this context, higher log-canonical) singularities and higher rational singularities emerged (see \cite{JKSY-DuBois}, \cite{MOPW}, \cite{FL-DuBois}, \cite{MP-rat}). In the first version of \cite{KL} (which appeared before the above-mentioned citations), we have noted that the higher Du Bois/rationality conditions lead to tighter control of the LMHS --- essentially, the central fiber determines the LMHS up to higher coniveau (Theorem \ref{th-klc}). Subsequently, other proofs and related statements were obtained by different methods (\cite{FL-DuBois} and \cite{MY}).

Let $f\colon \cx\to\Delta$ be a projective family, with smooth total space of dimension $n+1$ and reduced singular fiber (that is, $(f)=X_0$). Write $T=T^{\text{ss}}T^{\text{un}}=T^{\text{ss}}e^N$ for the monodromy about $\{0\}$, and $(-)^{T^{\text{ss}}}=(-)^{u}$ resp.~$(-)^n$ for unipotent resp.~non-unipotent parts. Since the non-unipotent parts of $H^k_{\lm}(X_t)^n$ and $H^k_{\van}(X_t)^n$ are isomorphic, the latter consists of $N$-strings centered about $p+q=k$. Moreover, by Clemens-Schmid, $H^k(X_0)$ surjects onto the $N$-invariants in the unipotent part $H^k_{\lm}(X_t)^u$; and so $H^k_{\van}(X_t)^u$ consists of $N$-strings centered about $p+q=k+1$. (As a MHS, $H^k_{\van}(X_t)$ is also symmetric under exchange of $p$ and $q$.)  

\subsection*{Rational and du Bois singularities}
Combining the vanishing cycle sequence \eqref{eq7.0} and its $T^{\text{ss}}$-invariants with the isomorphisms under $\sp$ in Theorems \textbf{I.9.3} and \textbf{I.9.11} (as well as the above symmetries) yields at once the following ``global'' analogue of Proposition \textbf{II.1.6}:

\begin{thm}\label{t7.1}
If $X_0$ has du Bois \textup{[}resp. rational\textup{]} singularities, then $(H^k_{\van}(X_t)^u)^{p,q}$ vanishes for $(p,q)$ outside the range $[\max\{1,k-n+1\},\min\{k,n\}]^{\times 2}$ \textup{[}resp. $[\max\{2,k-n+2\},\min\{k-1,n-1\}]^{\times 2}$\textup{]} while $(H^k_{\van}(X_t)^n )^{p,q}$ vanishes for $(p,q)$ outside the range $[\max\{1,k-n+1\},\min\{k-1,n-1\}]^{\times 2}$. In particular, the level of $H^k_{\pha}(X_0)$ is no more than $\min\{k-2,2n-k\}$ \textup{[}resp. $\min\{k-4,2n-k-2\}$\textup{]}.
\end{thm}
\[\includegraphics[scale=0.7]{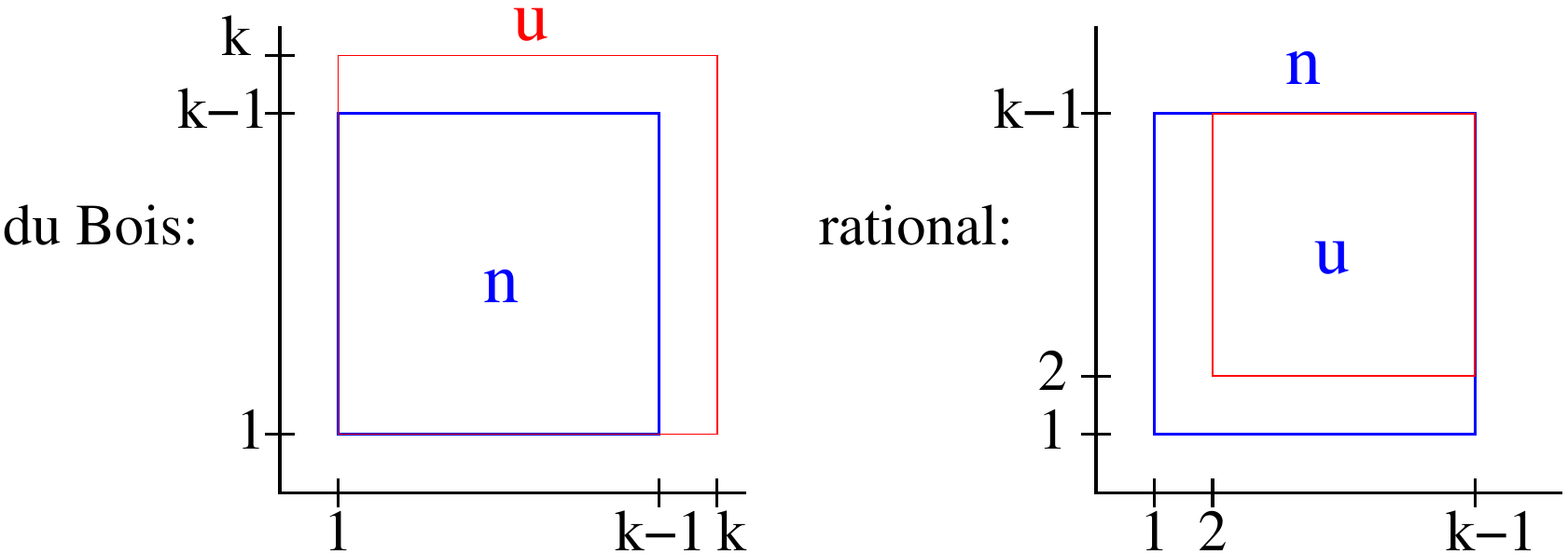} \]
\begin{proof}
More precisely, the above argument only gives the result when $\cx$ and $f$ admit algebraic extensions, and when $k\leq n$. The full statement follows immediately from \cite[(2.4.5) and (2.5.9)]{KLS}.
\end{proof}

\begin{rem}
(i) Note that in the rational case, the space\footnote{Recall that the \emph{phantom cohomology} $H^k_{\pha}(X_0)$ is the kernel of $\sp$ (or equivalently, image of $\delta$) in \eqref{eq7.0}.} $H^{1,k}_{\pha}(X_0)$ considered in \S\textbf{II.4} is outside the ``$u$'' range hence vanishes.

(ii) The result actually holds (by the proof of \cite[Thm.~3]{KLS}) as long as $\cx$ is an intersection homology manifold, i.e. $\IC_{\cx}=\QQ_{\cx}[n+1]$.
\end{rem}

Henceforth we assume that $f$ extends to a projective morphism of quasi-projective algebraic varieties.

\subsection*{Normal-crossing singularities and semistable reduction}
Though a normal-crossing variety has du Bois singularities, the term ``du Bois special fiber'' is typically used in the context where $X_0=(f)$, since this is when it implies ``cohomological insignificance'' of the degeneration (in the sense that $\gr^0_F H^k(X_0)\cong \gr_F^0 H_{\lm}^k(X_t)$). Suppose instead that $X_0 = \cup Y_i \subset \cx$ is a NCD with $\cx$ and $\{Y_i\}$ smooth, and $\text{ord}_{Y_i}(f)=a_i \in \mathbb{N}$ ($\implies \text{sing}(X_0)=\cup_{a_i >1} Y_i$ in Thm.~\textbf{I.5.5}). Calculating $\psi_f \QQ_{\cx}$ yields a short proof of the well-known

\begin{prop}[Clemens \cite{Clemens}]
The order of $T^{\text{ss}}$ on $H^*_{\lm}$ \textup{(}or $H^*_{\van}$\textup{)} divides $\text{lcm}\{a_i\}$.
\end{prop}

\begin{proof}
Let $z_1,\ldots,z_{n+1}$ be holomorphic coordinates on an open ball $U\subset\cx$ about $p\in X_0$, so that (renumbering the $V_i = U\cap Y_i$) $\ff:=f|_U = z_1^{a_1}\cdots z_r^{a_r}$. The Milnor fiber $\mathfrak{F}_{f,p}$ at $p$ has $N_p := \gcd\{a_1,\ldots,a_r\}$ connected components, which are cyclically permuted by $T$. Indeed, writing $U\overset{\ff}{\to}\Delta$ as a composition $U\overset{\rho}{\to}V\overset{\fg}{\to}\Delta$ where $\fg(x_1,\ldots,x_{n+1})=x_1\cdots x_k$ and $\rho(z_1,\ldots, z_{n+1})=(z_1^{a_1},\ldots,z_r^{a_r},z_{r+1},\ldots,z_{n+1})$, finiteness of $\rho$ gives $\imath_p^*\psi_{\ff}\QQ_U=\imath_{\uo}^*\psi_{\fg} \rho_* \QQ_U = \oplus_{j=0}^{N_p -1}\imath_{\uo}^* \psi_{\fg} \QQ_V$ (with $T^{\text{ss}}$ permuting factors) hence $H^k(\mathfrak{F}_{f,p})\otimes \CC = \oplus_{\ell=0}^{N_p -1}H^k(\mathfrak{F}_{\fg,\uo})\otimes \CC_{\zeta_{N_p}^{\ell}}$ (with $T^{\text{ss}}$ multiplying $\CC_{\chi}$ by $\chi$). Since $\HH^*_{\lm}=\HH^*(X_0,\psi_f\QQ_{\cx})$ and each $N_p \mid \text{lcm}\{a_i\}$, we are done.	
\end{proof}

\begin{rem}
To see what the $H^k(\mathfrak{F}_{\fg,\uo})= H^{k-n}(\imath^*_{\uo}\ppsi_{\fg}\QQ_V [n+1])$ look like for the semistable degeneration $\fg\colon V\to \Delta$ above, let $\tilde{V}\overset{\beta}{\to}V$ be the blow-up along $x_1 = \cdots = x_r =0$ and $D:=\beta^{-1}(\uo)\cong \PP^{n-r}$. Writing $((\CC^*)^{n-r}\cong)\,D^*\overset{\jmath}{\hookrightarrow}D$ for the complement of the proper transform of $\cup_{i=1}^r \{x_i =0\}$, we have\footnote{Here we use that $\psi_{\fg}$ depends only on the restriction of its argument to $g\neq 0$. Note that $\fg\circ \beta$ has order $r$ on $D^*$, so $\psi_{\fg\circ \beta}\QQ_V$ has rank $r$ there; but the unipotent part $\psi^u_{\fg\circ \beta}\QQ_V$ only has rank $1$.} 
\begin{equation*}
	\begin{split}
	\imath^*_{\uo}\ppsi_{\fg}\QQ_V[n+1] &= \imath^*_{\uo}\ppsi_{\fg}^u \QQ_V[n+1] \cong \imath_{\uo}^* \ppsi_{\fg}^u R\beta_* \QQ_{\tilde{V}}[n+1]\\
	&\cong \imath^*_{\uo}R\beta_* \ppsi^u_{\fg\circ\beta}\QQ_{\tilde{V}}[n+1]\cong R\Gamma_D \imath_D^* \ppsi^u_{\fg\circ \beta}\QQ_{\tilde{V}}[n+1]\\
	& \cong R\Gamma_D R\jmath_* \jmath^* \imath^*_D \ppsi^u_{\fg\circ \beta}\QQ_{\tilde{V}}[n+1] \cong R\Gamma_{D^*}\QQ_{D^*}[n]
	\end{split}
\end{equation*}
where we used \cite[(4.7)]{Sa-steen} for the penultimate isomorphism. Consequently
\begin{equation}\label{eq7.1}
H^k(\mathfrak{F}_{x_1\cdots x_r,\uo} )\cong H^k((\CC^*)^{n-r})\cong \QQ(-k)^{\oplus \binom{n-r}{k}}	
\end{equation}
as MHS's. If our original $f\colon \cx \to \Delta$ was semistable (locally $U=V$; all $a_i =1$), one deduces that $\ppsi_f\QQ_{\cx}[n+1]$ is a cosimplicial complex with term $\mathcal{I}^{[r]}_* \QQ_{X_0^{[r]}}\otimes H^*(\PP^r)$ in degree $r-n$, where $\mathcal{I}^{[r]}\colon X_0^{[r]}\,(=\amalg_{|I|=r+1}Y_I)\longrightarrow X_0$ denotes the normalization of the $(n-r)$-dimensional stratum of $X_0$. Replacing $H^*(\PP^r)$ by $\tilde{H}^*(\PP^r)$ yields $\pphi_f \QQ_{\cx}[n+1]$, as claimed in the proof of Theorem \textbf{I.6.4}.
\end{rem}

\begin{rem}\label{R7.1}
Suppose $\cy \overset{g}{\to}\Delta$ and $\cx \overset{f}{\to}\Delta$ are normal-crossing degenerations, with the first obtained from the second via base-change (by $\rho\colon t\mapsto t^{\ell}$) and birational modifications over $0$.  The \emph{torsion exponents} $\ell_j \in \mathbb{N}$ associated to this scenario are defined by the exact sequence
\begin{equation}\label{r1}
0\to \rho^* \RR^k f_* \Omega^{\bullet}_{\cx/\Delta}(\log X_0)\to \RR^k g_* \Omega^{\bullet}_{\cy/\Delta}(\log Y_0)\to \oplus_j \co/t^{\ell_j}\co \to 0.
\end{equation}
In this remark we explain how these relate to $\cv$-filtrations and (in some cases) spectra.

Writing $\HH:=R^k f^{\times}_*\CC_{\cx\setminus X_0}$ [resp. $\tilde{\HH}:=R^k g^{\times}_* \CC_{\cy\setminus Y_0} = \rho^* \HH$] for the local systems over $\Delta^*$, we may define lattices $\cv^{\beta} = \oplus_{\alpha\in[\beta,\beta+1)}\CC\{t\} C^{\alpha}\subset (\jmath_* (\HH\otimes \co_{\Delta^*}))_0$ [resp. $\tilde{\cv}^{\beta},\tilde{C}^{\alpha}$] as in the first two paragraphs of $\S$\textbf{II.5.2}.  By \cite[Proposition 2.20]{St2}, the stalks of the locally free sheaves $\RR^k f_* \Omega^{\bullet}_{\cx/\Delta}(\log X_0)$ and $\RR^k g_* \Omega_{\cy/\Delta}(\log Y_0)$ at $0$ identify with $\cv^0$ resp. $\tilde{\cv}^0$.  Since $\delta_{t^{\ell}}=\tfrac{1}{\ell}\delta_t$, base-change yields isomorphisms $\rho^* C^{\alpha}\overset{\cong}{\to} \tilde{C}^{\ell\alpha} = t^{\lfloor \ell\alpha\rfloor}\tilde{C}^{\{\ell\alpha\}}$ for each $\alpha \in [0,1)$.  Further, by [op. cit., (2.16-21)], we have a (non-canonical) isomorphism of $H^k_{\lm}(X_t)$ with
\begin{equation}\label{r2}
\RR^k f_* \Omega_{\cx/\Delta}(\log X_0)\otimes_{\co_{\Delta}}\co_{\{0\}}\cong \cv^0/t \cv^0 \cong \cv^0/\cv^1 = \oplus_{\alpha \in [0,1)} C^{\alpha}
\end{equation}
that identifies the $\mathbf{e}(-\alpha)$-eigenspace of $T^{\text{ss}}$ on $H^k_{\lm}$ with $C^{\alpha}$.  The upshot is that 
\begin{equation}\label{r3}
\oplus_j \co/t^{\ell_j}\co = \oplus_{\alpha\in [0,1)} \left( \co/t^{\lfloor \ell\alpha \rfloor}\co\right)^{\oplus \dim (H^k_{\lm}(X_t)_{\mathbf{e}(-\alpha)})};
\end{equation} 
in particular, if $\cy$ is semistable, then (for all $\alpha$ appearing) $\{\ell\alpha\}=0 \implies \lfloor \ell\alpha \rfloor =\ell\alpha$.  Refining this observation with respect to the Hodge filtration (viz., $\gr_F^p H^k_{\lm}(X_t)_{\mathbf{e}(-\alpha)}=\gr_F^p C^{\alpha}$) yields a short proof of \cite[Thm. D]{EFM}.

Now consider the special case where $(\cx,X_0)$ (with semistable reduction $\cy$) arises as a log-resolution of $(\fx,D)$, where $\fx\overset{F}{\to}\Delta$ also has smooth total space, and $D=F^{-1}(0)$ is reduced with a single isolated singularity at $x$ with spectrum $\sum_{\beta \in \QQ}m_{\beta}[\beta]$.  The nonzero eigenvalues $\alpha$ of $\delta_t$ appearing in \eqref{r2} [resp. torsion exponents in \eqref{r3}] are then the fractional parts $\{-\beta\}$ [resp. $\ell\{-\beta\}$] for those $\beta \in \QQ\setminus \ZZ$ with $m_{\beta}\neq 0$.  More precisely, we have (essentially by [loc. cit.]) that 
\begin{equation}\label{r4}
\frac{R^{n-p}g_* \Omega^p_{\cy/\Delta}(\log Y_0)}{\rho^* R^{n-p}f_* \Omega^p_{\cx/\Delta}(\log X_0)} \cong \oplus _{\beta \in (p,p+1)\cap \QQ} \left( \co /t^{\ell\{-\beta\}}\co\right)^{\oplus m_{\beta}}.
\end{equation}
To give a simple example, if $\fx$ is a degeneration of elliptic curves with type IV singular fiber (hence $D_4$ singularity at $x$), we can take $\cx$ to be its blowup at $x$, and $Y_0$ to be the (smooth) CM elliptic tail.  Since the spectrum of $D_4$ is $[\tfrac{2}{3}]+2[1]+[\tfrac{4}{3}]$, taking $n=p=1$ and $\ell=3$ gives $\text{RHS}\eqref{r4}=\co/t^2\co$.  Indeed, in local coordinates at a node of $X_0$ we have $f\sim u^3 v$, and $\omega\sim u^2 dv$ for the generator of $(R^0 f_*\Omega_{\cx/\Delta}^1 (\log X_0))_0$; clearly then $\rho^* \omega$ vanishes to order $2$ along $Y_0$, verifying \eqref{r4} in this case.
\end{rem}

\subsection*{$\mathbf{k}$-log-canonical singularities}
Returning to the reduced singular fiber setting, we address how Corollary \textbf{II.4.2} generalizes to the setting of non-isolated singularities.

\begin{thm}\label{th-klc}
If $X_0$ has $k$-log-canonical singularities \textup{(}$\mathcal{I}_r(X_0)=\co_{\cx}$ $\forall r\leq k$\textup{),} then\footnote{Note:  superscript $*$ indicates that the statements hold in all degrees of cohomology.} $\tilde{H}^*(\mathfrak{F}_{f,x})\subseteq F^{k+1}\tilde{H}^*(\mathfrak{F}_{f,x})$ \textup{(}$\forall x\in\mathrm{sing}(X_0)$\textup{)} and 
\begin{equation}\label{eq-th1.6}
\gr^i_F H^*(X_0)\cong \gr_F^i H_{\lm}^*(X_t) \textup{($0\leq i\leq k$)}.
\end{equation}
\end{thm}

\begin{proof}
Begin by observing that the associated weight-graded of the vanishing cycles MHM is semisimple perverse, so that 
\begin{equation}\label{eq-klc}
\gr^W(\pphi_f \QQ_{\cx}[n+1])\simeq \oplus_{\ell=0}^d \IC_{\overline{\mathsf{S}_{\ell}}}(\cv_{\ell}),
\end{equation}
where $\{\mathsf{S}_{\ell}\}$ is a stratification of $\text{sing}(X_0)$ (with $\ell=\dim(\mathsf{S}_{\ell})$, $d=\dim[\text{sing}(X_0)]\leq n-2$, and $\mathsf{S}_{\ell}$ smooth) and $\cv_{\ell}\to\mathsf{S}_{\ell}$ are some polarizable VHSs.  We claim that they satisfy $\cv_{\ell}=F^{k+1}\cv_{\ell}$; we say that $\cv_{\ell}$ \emph{is contained in} $F^{k+1}$.

Assuming this claim, the MHSs $H^r\imath_x^* \IC_{\overline{\mathsf{S}_{\ell}}}(\cv_{\ell})$ and $\HH^r(X_0,  \IC_{\overline{\mathsf{S}_{\ell}}}(\cv_{\ell}))=\IH^{r+\ell}(\overline{\mathsf{S}_{\ell}},\VV_{\ell})$ are contained in $F^{k+1}$ ($\forall r\in \ZZ$, $x\in X_0$).\footnote{In more detail: if $(\widetilde{\mathsf{S}}, D)\overset{\beta}{\twoheadrightarrow}(\overline{\mathsf{S}_{\ell}},\overline{\mathsf{S}_{\ell}}\setminus \mathsf{S}_{\ell})$ is a log-resolution, and $\widetilde{\mathsf{M}},\mathsf{M}$ are the ($F^{\bullet}$-filtered) $D$-modules underlying $\IC_{\widetilde{\mathsf{S}}}(\cv_{\ell}),\IC_{\overline{\mathsf{S}_{\ell}}}(\cv_{\ell})$, then $\mathrm{DR}^{\bullet}_{\widetilde{\mathsf{S}}}\widetilde{\mathsf{M}}$ is filtered quasi-isomorphic to $(\mathrm{DR}_{\widetilde{\mathsf{S}}}^{\bullet}\widetilde{\mathsf{M}})_{\log}:= \mathrm{DR}_{\widetilde{\mathsf{S}}}^{\bullet}\widetilde{\mathsf{M}} \cap (\omega^{\bullet}_{\widetilde{\mathsf{S}}}(\log D)\otimes \cv_{\ell,e})$ (with $\cv_{\ell,e}$ the canonical extension), cf.~\cite[p.~159]{Sa5}.  Since $\IC_{\overline{\mathsf{S}_{\ell}}}(\cv_{\ell})\subseteq R\beta_* \IC_{\widetilde{\mathsf{S}}}(\cv_{\ell})$ is a direct factor by the Decomposition Theorem, and $(\mathrm{DR}^{\bullet}_{\widetilde{\mathsf{S}}}\widetilde{\mathsf{M}})_{\log} =F^{k+1} (\mathrm{DR}^{\bullet}_{\widetilde{\mathsf{S}}}\widetilde{\mathsf{M}})_{\log}$, the assertions follow. (Alternatively, assuming wolog $\cv_{\ell}$ pure of weight $w_{\ell}$, one can use duality and $F^{w_{\ell}-k}\mathrm{DR}^{\bullet}_{\widetilde{\mathsf{S}}}\widetilde{\mathsf{M}}=\{0\}$.)}  Hence so are all the terms of the spectral sequences ${}^x E_1^{i,j}=H^{i+j} \imath^*_x \gr^W_{-i}\pphi_f \QQ_{\cx}[n+1]$ and  $E_1^{i,j}=\HH^{i+j}(X_0,\gr^W_{-i}\pphi_f\QQ_{\cx}[n+1])$, which converge to $H^{i+j}\imath_x^* \pphi_f \QQ_{\cx}[n+1]\cong \tilde{H}^{i+j+n}(\mathfrak{F}_{f,x})$ and $\HH^{i+j}(X_0,\pphi_f\QQ_{\cx}[n+1])\cong H_{\van}^{i+j+n}(X_t)$ respectively. Conclude that all $\tilde{H}^*(\mathfrak{F}_{f,x})\subset F^{k+1}$ and $H^*_{\van}(X_t)\subset F^{k+1}$, whence \eqref{eq-th1.6} follows from \eqref{eq7.0}.

To prove the claim, let $\tilde{\alpha}_f=\underset{x\in X_0}{\text{min}}\tilde{\alpha}_{f,x}$ denote the (global and local) microlocal log-canonical thresholds of $f$.\footnote{We do not discuss this here; see 
\cite[$\S$1.4]{KLS}.} By \cite[Cor.~2]{Sai16} (see also \cite[Cor.~C]{MP18ii}), the $k$-log-canonicity assumption implies that $k=\lfloor\tilde{\alpha}_f\rfloor -1$, hence that $\tilde{\alpha}_{f,x}\geq k+1$ $\forall x\in X_0$.  Consider any $x\in \mathsf{S}_0$, so that the only subobject of \eqref{eq-klc} supported at $x$ is the $\ell=0$ term $\imath^x_*\cv_{0,x}$ (where $\cv_{0,x}$ is a Hodge structure with $T^{\text{ss}}$-action).  In the the proof of \cite[Prop.~2.3]{Sai17}, Saito shows that its spectrum is supported in $[\tilde{\alpha}_{f,x},n+1-\tilde{\alpha}_{f,x}]$; in particular, $\cv_{0,x}=F^{k+1}\cv_{0,x}$.

Next consider a point $x\in \mathsf{S}_{\ell}$. Taking $\ell$ hyperplane sections of $\cx$ through $x$ (transverse to $\mathsf{S}_{\ell}$), the slice of $X_0$ still has $k$-log-canonical singularities at $x$ by \cite[(0.2)]{MP18i}.  The argument of the last paragraph applied to the slice yields that the spectrum of $\cv_{\ell,x}$ (as a Hodge structure with $T^{\text{ss}}$-action) lies in $[\tilde{\alpha}_{f,x},n+1-\ell-\tilde{\alpha}_{f,x}]$; in particular, we have $\cv_{\ell,x}=F^{k+1}\cv_{\ell,x}$.  So the $\{\cv_{\ell}\}$ are contained in $F^{k+1}$ and the claim is proved.
\end{proof}

\section{The Saito-Siersma-Steenbrink formula}\label{S7.2}

Consider as above a projective morphism $f\colon \cx \to \Delta$ from a smooth total space to the $t$-disk, which is smooth over $\Delta^*$, with $\dim(\text{sing}(X_0))=1$.  (That is, if $X_0$ is nonreduced, then $n=1$.) For any point $p\in\text{sing}(X_0)$, let $V_{f,p}^k :=H^k(\imath_p^* \phi_f \QQ_{\cx})$ be the $k^{\text{th}}$ (reduced) cohomology of the Milnor fiber, which has a $T^{\text{ss}}$-action and MHS (by \cite{Sai90b}) hence (weighted) spectra $\sigma_{f,p}^k , \tilde{\sigma}_{f,p}^k$ as in Definition \ref{def-spectrum}. On the open stratum $\sfs_1\subset\text{sing}(X_0)$,\footnote{This may be smaller than the smooth part of $\text{sing}(X_0)$; for instance, the pinch points studied below belong to $\sfs_0$ but also the smooth part of $\text{sing}(X_0)$.} these are nonzero only for $k=n-1$, and yield a VMHS $\cv_f^{n-1}$. The formula we now describe concerns the finite set $\sfs_0 = \text{sing}(X_0)\setminus \sfs_1$ of ``bad'' points, at which $V_{f,p}^{n-1}$ and $V_{f,p}^n$ may both be nonzero.

Restricting to a neighborhood $U\subset \cx$ of $p$, which (by choosing holomorphic coordinates $z_1,\ldots,z_{n+1}$ at $p$) we regard as an open ball in $\CC^{n+1}$, denote the restriction of $f$ by $\ff\colon U\to \Delta_t$. (In this context, we shall also denote $p$ by $\uo$.). Writing $Z=U\cap \text{sing}(X_0)=\cup_i Z_i$ as a union of (analytic) irreducible components, we assume that $U$ is chosen sufficiently small that (for each $i$) $Z_i \cap \sfs_0 = \{\uo\}$ and $Z_i^* :=Z_i\setminus \{\uo\}$ [resp. the normalization $\tilde{Z}_i$] is a punctured disk [resp. disk]. Let $\fg$ be the restriction to $U$ of a linear form on $\CC^{n+1}$ whose further restriction to each $Z_i$ is finite --- yielding a diagram
$$
\xymatrix{Z_i \ar [r]^{\fg_i} & \Delta_s \\ \tilde{Z}_i \ar [r]^{\tilde{\fg}_i}_{\cong} \ar @{->>} [u] & \Delta_{\tilde{s}_i} \ar [u]_{(\cdot)^{\mu_i}}}
$$
--- and assume in addition that the critical locus $\text{sing}(f|_{\fg^{-1}(0)})=\{\uo\}$.\footnote{Equivalently, the critical locus of $\pi$ intersects $\pi^{-1}(\Delta_t\times\{0\})$ only in $\{\uo\}$.} Together with $\ff$, this yields a holomorphic map $\pi=(\ff,\fg)\colon U\to \Delta^2_{t,s}$ to the bi-disk with discriminant locus $\Lambda = \cup_s \pi\{\text{sing}(\ff|_{\fg^{-1}(s)})\}$.  Each irreducible curve $C\subset \Lambda^{\circ}:=\overline{\Lambda \cap (\Delta^*)^2}$ has a Puiseux expansion $t=\gamma_C \cdot s^{\mathfrak{r}_C} + \{\text{higher-order terms}\}$, and we set $$\mathfrak{r}:=\max_{C\subset \Lambda^{\circ}}\{0,\mathfrak{r}_C\}\in \QQ_{\geq 0}.$$ For every $r\in \mathbb{N}$ with $r>\mathfrak{r}$, the Yomdin deformation $\ff+\fg^r$ has an \emph{isolated} singularity at $\uo$; in fact, we can even take $r=\mathfrak{r}$ provided none of the curves $C$ have $\mathfrak{r}_C=\mathfrak{r}$ and $\gamma_C = -1$ \cite{Sa-steen}.
\begin{figure}
\center{\includegraphics[scale=0.6]{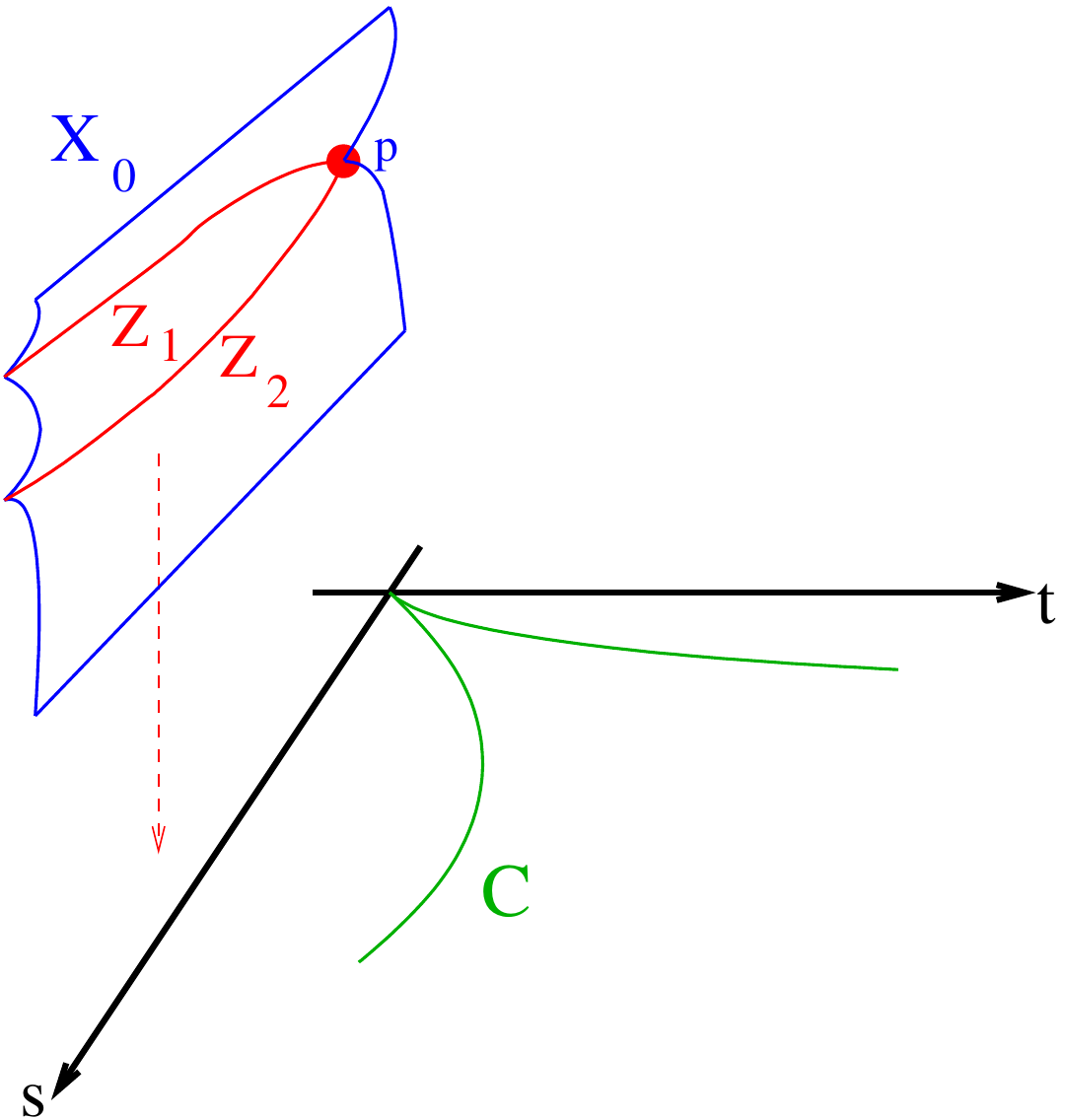}}
\caption{Each $\fg_i$ has finite degree $\mu_i$.}
\end{figure}

Writing $\cv_i$ for the restriction of $\cv_f^{n-1}$ to $Z_i^* \cong \tilde{Z}_i \setminus \{\uo\}$, its LMHS 
\begin{equation}\label{eq-vertlim}
\cv_i^{\lm}:=\psi_{\tilde{\fg}_i}\cv_i = \psi_{\tilde{\fg}_i}\ch^{-1}(\pphi_{\ff}\QQ_U [n+1])|_{Z_i^*}
\end{equation}
has commuting actions of ``horizontal'' monodromy $T_i$ (associated to $\phi_{\ff}$) and ``vertical'' mondromy $\tau_i$ (associated to $\psi_{\tilde{\fg}_i}$). Let $\tilde{\sigma}_{\lm,i}^{n-1} = \sum_j [(\alpha_{ij},w_{ij})]$ be the weighted spectrum of $(\cv^{\lm}_i,T_i^{\text{ss}})$; that is, we have a basis $\{v_{ij}\}\subset \cv^{\lm}_i \otimes \CC$ such that $v_{ij}\in (\cv^{\lm}_i)^{\lfloor\alpha_{ij}\rfloor,w_{ij}-\lfloor \alpha_{ij}\rfloor}$ and $T^{\text{ss}}_i v_{ij} = e^{2\pi\sqrt{-1}\alpha_{ij}}v_{ij}$. Further, we may choose these $\{ v_{ij}\}$ to be a simultaneous eigenbasis for $\tau_i^{\text{ss}}$, with eigenvalues $e^{2\pi\sqrt{-1}\beta_{ij}}$ for some $\{\beta_{ij}\}\subset [0,1)$.  Write formally $$\tilde{\tau}^0_{\lm,i}:=\sum_j \sum_{k=0}^{\mu_i r -1} [( \tfrac{\beta_{ij}+k}{\mu_i r},0 )]$$ for the corresponding ``spectra''; and define spectral ``convolutions'' by 
\begin{multline*}
\tilde{\sigma}_{\lm,i}^{n-1} \circledast \tilde{\tau}^0_{\lm,i} := \textstyle\sum_{j,k} [(\alpha_{ij},w_{ij})*(\tfrac{\beta_{ij}+k}{\mu_i r},0)] 
\\
\mspace{50mu} \text{resp.}\mspace{50mu}\sigma_{\lm,i}^{n-1}\circledast \tau_{\lm,i}^0 := \textstyle\sum_{j,k}[\alpha_{ij}+\tfrac{\beta_{ij}+k}{\mu_i r}],
\end{multline*}
where ``$*$'' is given\footnote{This is the same operation on weighted spectra as in \S\textbf{II.6} on the Sebastiani-Thom formula.} by $(\alpha,w)*(\beta,\omega):=(\alpha+\beta,w+\omega+\langle \alpha\mid\beta\rangle)$ with \begin{equation*}
\langle\alpha\mid\beta\rangle := \left\{ \begin{array}{cc} 0 & \alpha\text{ or }\beta\in \ZZ \\ 1 & \alpha,\beta,\alpha+\beta\notin \ZZ \\ 2 & \alpha,\beta\notin\ZZ,\, \alpha+\beta\in\ZZ. \end{array} \right.
\end{equation*}

\begin{thm}[SSS formula]\label{th-sss}
\textup{(i)}  With notations and assumptions as above, we have the equality of spectra
\begin{equation}\label{eqsss1}
\sigma_{f,p}^n - \sigma_{f,p}^{n-1} = \sigma_{\ff+\fg^r,\uo}^n - \sum_i \sigma_{\lm,i}^{n-1}\circledast \tau_{\lm,i}^0 .
\end{equation}
\textup{(ii)} If $r>\mathfrak{r}$ then the corresponding equality $\widetilde{\eqref{eqsss1}}$ of \emph{weighted} spectra holds.
\end{thm}

\begin{proof}[Sketch]
Though only (i) is stated in \cite{Sa-steen}, Saito's proof actually establishes (ii) as well (which only fails for $r=\mathfrak{r}$ due to the limit taken just after [op. cit.,(4.6.2)]).  The main step in his proof is to establish the following
\begin{lem}[\cite{Sa-steen}, Thm 2.5]
Given $\mathcal{M}\in\mathrm{MHM}(\Delta^2_{s,t})$ smooth over $(\Delta^*)^2\setminus \Lambda^{\circ}$, with $\Lambda^{\circ}$ tangent to the $t$-axis at $(0,0)$.  Then we have the equality of spectra\footnote{For $V^{\bullet}\in D^b\mathrm{MHS}$ with an action by $T^{\text{ss}}$, $\sigma(V^{\bullet})\in \ZZ[\QQ]$ means the alternating sum $\sum_m (-1)^m \sigma(V^m)$.}
\begin{equation}\label{eq!L}
\sigma(\imath^*_{\uo}\pphi_{s+t}\mathcal{M}) - \sigma(\imath^*_{\uo}\pphi_t \mathcal{M})= \sum_{\ell}[\alpha_{\ell}^{\mathcal{M}}+\beta_{\ell}^{\mathcal{M}}]	
\end{equation}
for the monodromy $T$ about $t=0$, where $\sum_{\ell}[\alpha^{\mathcal{M}}_{\ell}]=\sigma(\ppsi_s \pphi_t\mathcal{M})$ and the $\beta_{\ell}^{\mathcal{M}}\in [0,1)$ are $\tfrac{\log(\cdot)}{2\pi\sqrt{-1}}$ of the eigenvalues of monodromy $\tau$ about $s=0$.
\end{lem}
\noindent (The argument is similar to our proof of Theorem \textbf{II.6.1}.)  With this in hand, the assumption that $\text{sing}(\fg^{-1}_i(0))=\{\uo\}$ guarantees a suitable projective fiberwise compactification of $\pi$ to 
$\bar{\pi}\colon \bar{\cx}\to \Delta^2$, and defining $\rho\colon \Delta^2 \to \Delta^2$ by $(s,t)\mapsto (s^r,t)$, we set 
$$\mathcal{M}\,:=\, \pch^0 R(\rho\circ \bar{\pi})_* \QQ_{\cx}[n+1].$$
By finiteness of $Z$ over the $s$-axis, we then have (for $\mathsf{h}=s+t$ resp. $t$ and $\mathsf{H}=\fg^r + \ff$ resp. $\ff$) the middle equality of
\begin{equation*}
     \begin{split}
     \imath^*_{(0,0)} \pphi_{\mathsf{h}}\mathcal{M} 
     &= \imath^*_{(0,0)}\pch^0 R(\rho\circ \bar{\pi})_* \pphi_{\mathsf{H}}\QQ_{\cx}[n+1] 
     \\
     &= \imath^*_{(0,0)}R(\rho\circ\bar{\pi})_* \pphi_{\mathsf{H}}\QQ_{\bar{\cx}}[n+1] = \imath^*_{\uo}\pphi_{\mathsf{H}}\QQ_{\bar{\cx}}[n+1].
     \end{split}
\end{equation*}
Plugging this in to the LHS of \eqref{eq!L} gives $\sigma^n_{\ff+\fg^r,\uo}-(\sigma_{f,p}^n - \sigma_{f,p}^{n-1})$. The $\{\alpha_{\ell}^{\mathcal{M}},\beta_{\ell}^{\mathcal{M}}\}$ come from taking push-forwards of the various $\cv_i$ under the finite maps $\Delta^*\cong Z^*_i \overset{\rho\circ\bar{\pi}}{\longrightarrow}\Delta^*$ of degree $\mu_i r$, whereupon each pair $\{\alpha_{ij},\beta_{ij}\}$ becomes $\bigcup_{k=0}^{\mu_i r -1} \{\alpha_{ij},\tfrac{\beta_{ij}+k}{\mu_i r}\}$.
\end{proof}

\begin{rem}
Combining Theorem \ref{th-sss}(i) with toric geometry techniques ($\S$\textbf{II.5}) yields a combinatorial formula for LHS\eqref{eqsss1} in the (nonisolated) Newton-nondegenerate simplicial case; cf. \cite{JKSY2} ($n=3$) and \cite{Sai20} ($n$ arbitrary), which appeared while an earlier version of this paper was being prepared.  While this formula applies in principle to some of the examples below, it does not yield the weighted spectra or appear that it would simplify the computations.
\end{rem}

Applying part (ii) of Theorem \ref{th-sss} can be a little tricky in practice, as one needs to correctly compute the ``vertical'' LMHSs \eqref{eq-vertlim}.  Here is an instructive

\begin{example}\label{ex-key}
Suppose $f$ has local form $\ff=x^2y^2+(x^4+y^4)z+x^5+y^5$, where we have relabeled $(z_1,z_2,z_3)=:(x,y,z)$.  Then $Z$ is the portion of the $z$-axis inside the small ball $U$ about $\uo$, and taking $\fg=z$ yields $\mathfrak{r}=5$; so $r=7$ will work.\footnote{Since there is only one component of $Z$ and it is normal, the sum over $i$ is dropped, and $\mu_1=1$. Taking $r$ odd (we could have taken $r=6$) simplifies the computation.}  The VMHS $\cv$ on $Z^* \underset{\fg}{\overset{\cong}\to} \Delta^*$ is given, at each $s\in \Delta^*$, by the vanishing-cycles MHS on $H^1$ of the Milnor fiber of $\ff_s=x^2y^2+(x^4+y^4)s+x^5+y^5$, which is a curve singularity of ordinary quadruple-point type.  The weighted spectrum and Hodge-Deligne diagram are thus
\begin{equation}\label{eq-Vfs}
\pgfmathsetmacro\MathAxis{height("$\vcenter{}$")}
\begin{matrix}\tilde{\sigma}^1_{\ff_s}=[(\tfrac{1}{2},1)]+2[(\tfrac{3}{4},1)]+3[(1,2)]\\[0.5em] \mspace{40mu}+2[(\tfrac{5}{4},1)]+[(\tfrac{3}{2},1)]\end{matrix}
\mspace{50mu}
\begin{tikzpicture}[scale=0.8,baseline={(0, 0.5cm+\MathAxis pt)}]
\draw [thick,gray,<->] (0,2) -- (0,0) -- (2,0);
\node at (2.1,0) {\tiny $p$};
\node at (0.2,1.9) {\tiny $q$};
\filldraw (1.3,1.3) circle (3pt);
\filldraw (0,1.3) circle (3pt);
\filldraw (1.3,0) circle (3pt);
\node at (1.6,1.3) {\tiny $3$};
\node at (0.3,1.3) {\tiny $3$};
\node at (1.6,0.2) {\tiny $3$};
\end{tikzpicture}
\cv
\mspace{50mu}
\end{equation}
Note that the fractional parts of the first entries tell us the eigenvalues of the \emph{horizontal} monodromy $T^{\text{ss}}$.

To compute its LMHS $\cv^{\lim}$ at $s=0$, we must identify the action of vertical monodromy $\tau=\tau^{\text{ss}}e^{N_{\tau}}$ on $\cv$.  For this purpose, we can drop the $x^5+y^5$ part of $\ff_s$ and look at the ``tail'', consisting of $\cE_s:=\{(x^4+y^4)s+x^2y^2=w^4\}\subset \PP^2$ minus the 4 points $E_s:=\cE_s\cap \{w=0\}$; recall from \S\textbf{II.2} that $H^1(\cE_s\setminus E_s)\cong \cv_s$.  The three $(1,1)$-classes in \eqref{eq-Vfs} correspond to degree-0 divisors supported on $E_s$, on which $\tau$ has eigenvalues $-1$ (mult.~2) and $1$ (mult.~1).  The compact part $\cE_s$ splits into two irreducible components meeting in two tacnodes at $s=0$, which allows to calculate its LMHS using the vanishing-cycle sequence.  In fact, since the action of the automorphism $\mu\colon w\mapsto \mathbf{i}w$ on $H^1(\cE_s)$ encodes $T^{\text{ss}}$, we can use the ``eigenspectrum'' calculus of \cite{CDKP} to find a simultaneous Jordan basis for $T$ and $\tau$.  This yields, for comparison to \eqref{eq-Vfs},
\begin{equation}\label{eq-Vflim}
\pgfmathsetmacro\MathAxis{height("$\vcenter{}$")}
\begin{matrix}\tilde{\sigma}^1_{\lim}={\color{blue}[(\tfrac{1}{2},0)]}+2[(\tfrac{3}{4},1)]+3[(1,2)]\\[0.5em] \mspace{40mu}+2[(\tfrac{5}{4},1)]+{\color{blue}[(\tfrac{3}{2},2)]}\end{matrix}
\mspace{30mu}
\begin{tikzpicture}[scale=0.8,baseline={(0, 0.5cm+\MathAxis pt)}]
\draw [thick,gray,<->] (0,2) -- (0,0) -- (2,0);
\node at (2.1,0) {\tiny $p$};
\node at (0.2,1.9) {\tiny $q$};
\filldraw (1.45,1.15) circle (3pt);
\filldraw (0,1.3) circle (3pt);
\filldraw (1.3,0) circle (3pt);
\filldraw [blue] (0,0) circle (3pt);
\filldraw [blue] (1.3,1.3) circle (3pt);
\node at (1.7,1.15) {\tiny $3$};
\node at (0.3,1.3) {\tiny $2$};
\node at (1.6,0.2) {\tiny $2$};
\node [blue] at (1,1.5) {\tiny $1$};
\node [blue] at (-0.3,0.2) {\tiny $1$};
\draw [thick,blue,->] (1.1,1.1) -- (0.2,0.2);
\node [blue] at (0.9,0.5) {\tiny $N_{\tau}$};
\end{tikzpicture}
\cv^{\lim}
\mspace{30mu}
\end{equation}
More precisely, the full ``eigenspectrum'' $\sum_{i=1}^9 [(\alpha_j,w_j,\beta_j)]$ (notation defined after \eqref{eq-vertlim}) is
\begin{equation*}
[(\tfrac{1}{2},0,0)]+2[(\tfrac{3}{4},1,\tfrac{3}{4})]+[(1,2,0)]+2[(1,2,\tfrac{1}{2})]\\ +2[(\tfrac{5}{4},1,\tfrac{1}{4})]+[(\tfrac{3}{2},2,0)],
\end{equation*}
in which $N_{\tau}$ ``connects'' the first and last terms.  From this, we can read off the spectral convolution $\tilde{\sigma}^1_{\lim}\circledast \tilde{\tau}^0_{\lim}=$
\begin{equation}\label{eq-spcon}
\begin{split}
[(\tfrac{1}{2},0)]&+ {\color{blue}[(\tfrac{9}{14},1)]+[(\tfrac{11}{14},1)]}+2[(\tfrac{6}{7},2)]+{\color{blue}[(\tfrac{13}{14},1)]}+[(1,2)]\\
&+2[(1,3)]+{\color{blue}[(\tfrac{15}{14},1)]}+2[(\tfrac{15}{14},2)]+3[(\tfrac{8}{7},2)]+{\color{blue}[(\tfrac{17}{14},1)]}\\
&+2[(\tfrac{17}{14},2)]+5[(\tfrac{9}{7},2)]+{\color{blue}[(\tfrac{19}{14},1)]}+2[(\tfrac{19}{14},2)]+5[(\tfrac{10}{7},2)]\\
&+3[(\tfrac{3}{2},2)]+5[(\tfrac{11}{7},2)]+2[(\tfrac{23}{14},2)]+{\color{blue}[(\tfrac{23}{14},3)]}+5[(\tfrac{12}{7},2)]\\
&+2[(\tfrac{25}{14},2)]+{\color{blue}[(\tfrac{25}{14},3)]}+3[(\tfrac{13}{7},2)]+2[\tfrac{27}{14},2)]+{\color{blue}[(\tfrac{27}{14},3)]}
\\
&+2[(2,3)]+{\color{blue}[(\tfrac{29}{14},3)]}+2[(\tfrac{15}{7},2)]+{\color{blue}[(\tfrac{31}{14},3)]+[(\tfrac{33}{14},3)]}.
\end{split}	
\end{equation}
Had we (incorrectly) used \eqref{eq-Vfs} instead of \eqref{eq-Vflim}, the weights in the blue terms of \eqref{eq-spcon} would all have been $2$.\footnote{In addition, $[(\tfrac{1}{2},0)]$ would have been $[(\tfrac{1}{2},1)]$, and one of the $[(\tfrac{3}{2},2)]$'s would have been $[(\tfrac{3}{2},1)]$.}

The weighted spectrum of the Yomdin deformation $\ff+\fg^7=x^2y^2+(x^4+y^4)z+x^5+y^5+z^7$ can be computed by the combinatorial methods of \S\textbf{II.5.1}:\footnote{In particular, Theorem \textbf{II.5.4} yields $N$-strings in its vanishing cohomology from $(p,q)=(1,2)$ to $(0,1)$ with $T^{\text{ss}}$-eigenvalues $\{e^{2\pi\mathbf{i}(\frac{1}{2}+\frac{k}{7})}\}_{k=1}^6$, arising from the edge in the Newton polygon from $(0,0,7)$ to $(2,2,0)$; notice that the weights are $1$ and $3$.  These and their conjugates are what will cancel with the blue terms of \eqref{eq-spcon}.}
we have $\tilde{\sigma}^2_{\ff+\fg^7,\uo}=$
\begin{equation*}
\begin{split}
&\phantom{+\;}[(\tfrac{9}{14},1)]+[(\tfrac{11}{14},1)]+2[(\tfrac{6}{7},2)]+[(\tfrac{13}{14},1)]\\
&+2[(1,3)]+[(\tfrac{15}{14},1)]+2[(\tfrac{15}{14},2)]+3[(\tfrac{8}{7},2)]+[(\tfrac{17}{14},1)]\\
&+2[(\tfrac{17}{14},2)]+5[(\tfrac{9}{7},2)]+[(\tfrac{19}{14},1)]+2[(\tfrac{19}{14},2)]+5[(\tfrac{10}{7},2)]\\
&+{\color{red}4}[(\tfrac{3}{2},2)]+5[(\tfrac{11}{7},2)]+2[(\tfrac{23}{14},2)]+[(\tfrac{23}{14},3)]+5[(\tfrac{12}{7},2)]\\
&+2[(\tfrac{25}{14},2)]+[(\tfrac{25}{14},3)]+3[(\tfrac{13}{7},2)]+2[\tfrac{27}{14},2)]+[(\tfrac{27}{14},3)]
\\
&+2[(2,3)]+[(\tfrac{29}{14},3)]+2[(\tfrac{15}{7},2)]+[(\tfrac{31}{14},3)]+[(\tfrac{33}{14},3)].
\end{split}	
\end{equation*}
Applying the SSS formula, we get
\begin{equation}\label{eq-saitoex}
\begin{split}
\tilde{\sigma}^2_{f,\uo}-\tilde{\sigma}^1_{f,\uo}&=\tilde{\sigma}^2_{\ff+\fg^7}-\tilde{\sigma}^1_{\lim}\circledast \tilde{\tau}^0_{\lim}\\
&=[(\tfrac{3}{2},2)]-[(\tfrac{1}{2},0)]-[(1,2)],
\end{split}
\end{equation}
in which taking the limit in \eqref{eq-Vflim} was evidently crucial in canceling the fractional terms. (After all, the spectrum of $f$ should be independent of the choice of $r$!)  With a bit more work, one can show that $\text{rk}(H^1(\mathfrak{F}_{f,\uo}))=2$ so that there is no cancelation in RHS\eqref{eq-saitoex}.
\end{example}

Before turning to further computations, we should mention one other tool which is sometimes useful in determining whether there is cancellation on the LHS of \eqref{eqsss1}.  If $\mathscr{F}$ resp. $\mathscr{G}$ are germs of analytic functions on $\CC^{m+1}$ resp. $\CC^{n+1}$, with Milnor fibers $\mathfrak{F}_{\mathscr{F}}, \mathfrak{F}_{\mathscr{G}}$ and join $\mathscr{F}\oplus\mathscr{G}$, then we have (as vector spaces)
\begin{equation}\label{eq-ST*}
H^{k+1}(\mathfrak{F}_{\mathscr{F}\oplus\mathscr{G}})\cong \oplus_{i+j=k}\tilde{H}^i (\mathfrak{F}_{\mathscr{F}})\otimes \tilde{H}^j(\mathfrak{F}_{\mathscr{G}})	
\end{equation}
with monodromy $T$ on the left induced by $T_1\otimes T_2$ on the right. See \cite{Sakamoto74} and \cite{Nemethi91}.

\section{Some non-isolated slc surface singularities}\label{S7.3}

One of the main motivations of our study is to use Hodge theory to understand KSBA compactifications (see \cite{Kollar-book}) for surfaces of general type as proposed by Griffiths and his collaborators (e.g. \cite{Griffiths}). The singularities occurring for varieties at the boundary of the KSBA compactification are semi-log-canonical (slc) and consequently Du Bois (cf. \cite{KK}). Thus, as discussed above, there is a tight connection between the frontier of the Hodge diamond of $X_0$ and that of the LMHS.  

By \eqref{eq7.0}, the remaining ingredient for determining the form of $H^*_{\lim}$ is the vanishing cohomology.  Here one needs more detailed input from the singularity types of $X_0$.  Fortunately, in the case of 2-dimensional slc hypersurface singularities, an explicit classification exists (cf. \cite{ksb}, \cite{LR}), and we can use Theorem \ref{th-sss} (with $n=2$) to determine $\sigma_{f,p}^*$ and thus $H^*_{\van}(X_t)$.  The spectra are summarized in Table \ref{T-X}.  We shall first describe briefly how they are computed, then how to get from them to $H^*_{\van}$, and finally apply the results to degenerations of $K3$ surfaces.


All of the singularities in Table \ref{T-X} have in common that on the open stratum $\mathsf{S}_1 \subset \text{sing}(X_0)$, the singularity type is $A_{\infty}$ (locally analytically, two components crossing normally along a curve).  Hence $\cv^1_f \cong \cl(-1)$ for $\cl$ a rank-one, type $(0,0)$ isotrivial VHS on $\mathsf{S}_1$ with (vertical) monodromies $\pm 1$. It follows that all sums over $j$ (and in most cases, over $i$) disappear, all $(\alpha_i,w_i)=(1,2)$, and all $\beta_i = 0$ or $\tfrac{1}{2}$; in Table \ref{T-X}, $\beta=\tfrac{1}{2}$ occurs only for the pinch point $D_{\infty}$.  Moreover, the local form of $f$ at each $p\in \mathsf{S}_0$ is such that $Z$ is a union of coordinate axes, and so all $\mu_i =1$.

Accordingly, the SSS formula reads (for $r>\mathfrak{r}$)
\begin{equation}\label{eqsss2}
\tilde{\sigma}_{f,p}^2 = \tilde{\sigma}_{f,p}^1 + \tilde{\sigma}_{\ff+\fg^r,\uo}-\sum_i\sum_{k=0}^{r-1}[(1+\tfrac{\beta_i + k}{r},2)].\end{equation}
One can readily compute $\tilde{\sigma}_{\ff+\fg^r,\uo}^2$ with SINGULAR (which we used for some entries in Table \ref{T-X}), or by hand (see the $J_{k,\infty}$ example below). So on the RHS of \eqref{eqsss2}, this leaves $\tilde{\sigma}^1_{f,p}$, i.e. the computation of $V^1_{f,p}\cong H^1(\mathfrak{F}_{f,p})$ as a MHS and $T^{\text{ss}}$-module. For $T_{\infty,\infty,\infty}$, $\S$\ref{S7.1} gives $V^1_{f,p}\cong \QQ(-1)^{\oplus 2}$; and we can directly show $\text{rk}(V^1_{f,p})=1$ [resp. $0$] for $T_{p,\infty,\infty}$ [resp. $T_{p,q,\infty}$] by fibering $\mathfrak{F}_{f,p}$ over the $x$-coordinate.  In other cases, $f=x^2 + F(y,z)$ is a suspension and \eqref{eq-ST*} yields $V^1_{f,p} = \tilde{H}^0 (\mathfrak{F}_{x^2})\otimes \tilde{H}^0(\mathfrak{F}_{F,\uo})$. If $\mathfrak{F}_{F,\uo}$ is connected ($D_{\infty},T_{2,q,\infty},J_{k,\infty}$), this is zero; while for $T_{2,\infty,\infty}$, $\mathfrak{F}_{F,\uo}$ has 2 components so that $V^1_{f,p}$ has rank $1$, and $T=T_1\otimes T_2$ acts by $(-1)^2=1$.\footnote{For this and the other rank 1 case ($T_{p,\infty,\infty}$), an easy way to deduce that $\tilde{\sigma}^1_{f,p}=[(1,2)]$ ($\implies V^1_{f,p}\cong \QQ(-1)$) is from the need to cancel a $-[(1,2)]$ in the third term of RHS\eqref{eqsss2} (that is not cancelled by the second term).}

\small
\begin{table}[ht]\SMALL
\caption{Nonisolated slc hypersurface singularities ($n=2$)}
\centering
\begin{tabular}{c|c|c|c|c}
symbol & local form of $f$ & $\fg,\mathfrak{r},N$ & $\tilde{\sigma}^1_{f,p}$ &  $\tilde{\sigma}^2_{f,p}$ \\ [0.5ex]
\hline 
&&&& \\
$A_{\infty}$ & $x^2+y^2$ & $z,0,1$ & $[(1,2)]$ & $0$ \\ &&&& \\
$D_{\infty}$ & $x^2 + y^2 z$ & $z-y,3,1$ & $0$ & $[(\tfrac{3}{2},2)]$ \\ &&&& \\
$T_{2,\infty,\infty}$ & $x^2+y^2z^2$ & $z-y,4,2$ & $[(1,2)]$ & $[(\tfrac{3}{2},2)]+[(2,4)]$ \\ &&&& \\
$T_{2,q,\infty}$ & $x^2+y^2z^2+y^q$ & $z,\tfrac{2q}{q-2},1$ & $0$ & $[(\tfrac{3}{2},2)]+[(2,4)]$ \\ & \;\;\;\;\;\;\;\;\;\;($q\geq 3$)&&&\;\;\;\;\;\;\;\;\;\;$+\sum_{\ell=1}^{q-1}[(1+\tfrac{\ell}{q},2)]$\\
$T_{\infty,\infty,\infty}$ & $xyz$ & $x+y+z,3,3$ & $2[(1,2)]$ & $[(2,4)]$ \\ &&&& \\
$T_{p,\infty,\infty}$ & $xyz+x^p$ & $y+z,\tfrac{2p}{p-1},2$ & $[(1,2)]$ & $\sum_{\ell=1}^{p-1}[(1+\tfrac{\ell}{p},2)]+[(2,4)]$ \\ &\;\;\;\;\;\;\;\;\;\;($p\geq 3$)&&& \\
$T_{p,q,\infty}$ & $xyz+x^p+y^q$ & $z,\tfrac{pq}{pq-p-q},1$ & $0$ & $\sum_{\ell=1}^{p-1}[(1+\tfrac{\ell}{p},2)]+[(2,4)]$ \\ &\;\;\;\;($q\geq p\geq 3$)&&& \;\;\;\;\;\;\;\;\;\;$+\sum_{\ell=1}^{q-1}[(1+\tfrac{\ell}{q},2)]$ \\
[1ex]
\end{tabular}
\label{T-X}
\end{table}\normalsize

As previously mentioned, Table \ref{T-X} makes use of a classification by Liu and Rollenske \cite{LR}. The point $p$ is simply $(0,0,0)$ in the coordinates used there, and $\tilde{\sigma}_{f,p}^2$ is calculated using \eqref{eqsss2}; $N$ denotes the number of components of $Z$. For the computation of $\mathfrak{r}$, consider for instance $D_{\infty}$: although $Z=\{x=y=0\}$, we cannot take $\fg=z$ (since then $\mathrm{sing}(\pi)\cap \pi^{-1}(\{ s=0\})=\{x=0=z\}\neq \{\uo\}$).  But $\fg=z-y$ works, and the critical locus of $\pi=(x^2+y^2z,z-y)$ is $\{x=y=0\}\cup\{x=2z+y=0\}$, with image $\{t=0\}\cup\{t=\tfrac{4}{27}s^3\}$. So $\mathfrak{r}=3$, and taking $r=4$ (and $\beta=\tfrac{1}{2}$) the RHS of \eqref{eqsss2} reads
$$0+\{[\tfrac{9}{8}]+[\tfrac{11}{8}]+[\tfrac{3}{2}]+[\tfrac{13}{8}]+[\tfrac{15}{8}]\}-\{[\tfrac{9}{8}]+[\tfrac{11}{8}]+[\tfrac{13}{8}]+[\tfrac{15}{8}]\}$$
(with all weights $=2$), giving the result in the table.

The results in the table allow us to compute the vanishing cohomology via the hypercohomology spectral sequence
\begin{equation}\label{eq-SS*}
E_2^{i,j}=H^i(\ch^j \phi_f\QQ_{\cx})\;\underset{i+j=*}{\implies}\;\HH^*(\phi_f\QQ_{\cx})=H^*_{\van}(X_t),	
\end{equation}
where the (non-perverse) cohomology sheaves record the reduced cohomologies of Milnor fibers
\begin{equation*}
\imath_p^* \ch^j\phi_f \QQ_{\cx}\cong H^j\imath_p^*\phi_f\QQ_{\cx} \cong H^j(\mathfrak{F}_{f,p})\cong V^j_{f,p}.
\end{equation*}
Recall that $\ch^2\phi_f\QQ_{\cx}$ is supported on the finite set $\mathsf{S}_0$, $\ch^1\phi_f\QQ_{\cx}$ is supported on $\bar{\mathsf{S}}_1 = \text{sing}(X_0)$, while the other cohomology sheaves vanish.  It follows that terms of \eqref{eq-SS*} other than $E_2^{0,1}$, $E_2^{1,1}$, $E_2^{2,1}$, and $E_2^{0,2}$ vanish, and the only possibly nonzero differential is $d_2\colon E_2^{0,2}\to E_2^{2,1}$. To describe $\ch^1\phi_f\QQ_{\cx}$ as a sheaf, note that $\left.(\ch^1\phi_f\QQ_{\cx})\right|_{\mathsf{S}_1} = \cv^1_f$ (viewed as a local system) and $\left.(\ch^1\phi_f\QQ_{\cx})\right|_{\mathsf{S}_0}=\oplus_{p\in\mathsf{S}_0}V_{f,p}^1$. If the $\{V_{f,p}^1\}$ vanish, then $\ch^1\phi_f\QQ_{\cx}\cong \jmath_! \cv_f^1$ under the inclusion $\jmath\colon \mathsf{S}_1\hookrightarrow \bar{\mathsf{S}}_1$; while in other cases the $\{\cv_{f,p}^1\}$ glue together with $\cv_f^1$ to yield the constant sheaf $\QQ_{\bar{\mathsf{S}}_1}(-1)$ ($T_{2,\infty,\infty},T_{p,\infty,\infty}$) or something more exotic ($T_{\infty,\infty,\infty}$; see Example \ref{ex7(1)} below).

We now turn to several examples involving degenerations of $K3$ surfaces. It is instructive to begin by looking at something familiar in this context.

\begin{example}\label{ex7(1)}
Suppose $\cx$ is a Kulikov type III semistable degeneration of $K3$'s with $\mathtt{F}$ components	, while $\bar{\mathsf{S}}_1$ consists of $\mathtt{E}$ $\PP^1$'s and $\mathsf{S}_0$ of $\mathtt{V}$ points (of type $T_{\infty,\infty,\infty}$).  As the dual graph is a triangulation of $S^2$, $\mathtt{F}=\mathtt{E}-\mathtt{V}+2$. Write $\mathcal{I}\colon \widetilde{\bar{\mathsf{S}}_1}\to X_0$ for the normalization of $\bar{\mathsf{S}}_1$ and $\imath\colon \mathsf{S}_0\hookrightarrow X_0$; then from the table, $\ch^2\phi_f \QQ_{\cx}\cong \imath_* \QQ_{\mathsf{S}_0}(-2)$ and $\ch^1\phi_f\QQ_{\cx}\cong\ker\{\mathcal{I}_*\QQ_{\widetilde{\bar{\mathsf{S}}_1}}(-1)\twoheadrightarrow\imath_*\QQ_{\mathsf{S}_0}(-1)\}$.  This yields
$$\xymatrix{\ch^2 & \QQ(-2)^{\oplus\mathtt{V}}\ar [rrd]^{d_2} \\ \ch^1 & \QQ(-1)^{\oplus \mathtt{F}-1}& \QQ(-1)& \QQ(-2)^{\oplus \mathtt{E}} \\ & H^0 & H^1 & H^2 \ar@{-}(10,-25);(90,-25)
\ar@{-}(10,-25);(10,8)}$$
for \eqref{eq-SS*}; since $H^3_{\van}=\ker\{H^4(X_0)\twoheadrightarrow H^4_{\lm}\}$ has rank $\mathtt{F}-1=\mathtt{E}-(\mathtt{V}-1)$, we must have\footnote{One can avoid a nontrivial $d_2$ by using the alternate hypercohomology spectral sequence ${}'E_1^{i,j}=H^i(X_0,\mathcal{K}^j) \implies H^*_{\van}(X_0)$ with $\phi_f\QQ_{\cx}\simeq\mathcal{K}^\bullet := \{\imath_*\QQ_{\mathsf{S}_0}(-2)[-2]\to \mathcal{I}_* R\Gamma\QQ_{\widetilde{\bar{\mathsf{S}}_1}}(-1)\to \imath_*\QQ_{\mathsf{S}_0}(-1)\}$. (Note that $\psi_f\QQ_{\cx}$ is \emph{not} quasi-isomorphic to $\oplus_j \ch^j\phi_f\QQ_{\cx}[-j]$; otherwise $d_2$ in \eqref{eq-SS*} would indeed be zero.)} 
$\mathrm{rk}(d_2)=\mathtt{V}-1$.   Conclude that $H^1_{\van}\cong \QQ(-1)^{\oplus\mathtt{F}-1}$ and $H^2_{\van}$ is an extension of $\QQ(-2)$ by $\QQ(-1)$, which via the vanishing cycle sequence
$$0\to H^1_{\van}\to H^2(X_0)\to H^2_{\lm} \to H^2_{\van} \to 0$$
yields that $H^2(X_0)$ is an extension of $\QQ(-1)^{\oplus 18+\mathtt{F}}$ by $\QQ(0)$.  For instance, this gives $h^2(X_0)=23$ if $\mathtt{F}=4$, which is borne out by applying Mayer--Vietoris to (say) the minimal semistable reduction of the standard tetrahedral degeneration of quartic $K3$'s.
\end{example}

\begin{example} \label{ex7(2)}
Next we consider a degeneration of $K3$'s with $\bar{\mathsf{S}}_1\cong \PP^1$ and $\mathsf{S}_0 = 4$ pinch points ($D_{\infty}$), about each of which the rank-one local system $\cl$ has monodromy $(-1)$. In particular, this gives $\ch^1\phi_f\QQ_{\cx}\cong \jmath_!\cl(-1)\cong \jmath_* \cl(-1)$, whose $H^1$ is $\IH^1(\cl)(-1)\cong H^1(E)(-1)$ with $E\overset{2:1}{\twoheadrightarrow} \bar{\mathsf{S}}_1$ the elliptic curve branched over $\mathsf{S}_0$. The spectral sequence \eqref{eq-SS*} is thus simply
$$\xymatrix{\ch^2 & \QQ(-1)^{\oplus4}_{-1} \\ \ch^1 & 0 & H^1(E)(-1) & 0  \\ & H^0 & H^1 & H^2 \ar@{-}(10,-25);(90,-25)
\ar@{-}(10,-25);(10,8)} $$	
where the subscript ``$-1$'' refers to the action of $T^{\text{ss}}$, and $H^2_{\van}$ is the direct sum\footnote{As usual, the extension is split by the action of $T^{\text{ss}}$.} of the two nonzero terms.

Now in the vanishing cycle sequence
\begin{equation}\label{eq-7E2}
0\to H^2(X_0)\to H^2_{\lm} \to H^2_{\van}\to H^3(X_0)\to 0
\end{equation}
we have two possible scenarios, depending on whether $H^3(X_0)$ is \emph{(a)} $0$ or \emph{(b)} $H^1(E)(-1)$. In case \emph{(a)}, the degeneration has type II (in the sense of Definition \textbf{II.5.4}) and the nonzero terms of \eqref{eq-7E2} have Hodge--Deligne diagrams
\[
\includegraphics[scale=0.65]{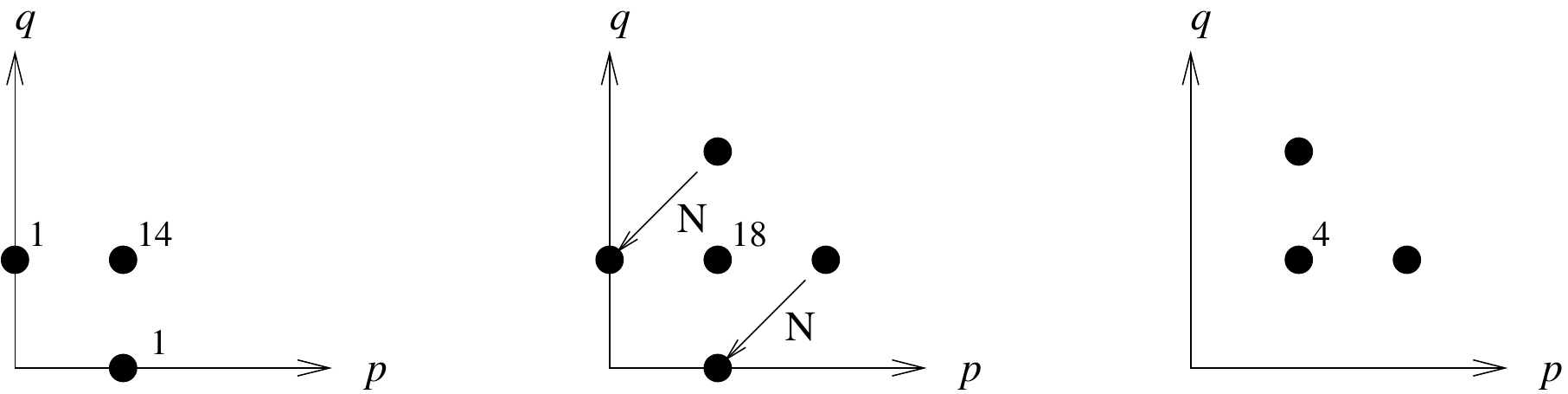}
\]
In case \emph{(b)}, the degeneration would be of type I, with the first two terms of \eqref{eq-7E2} pure. In fact, case \emph{(b)} cannot occur:  since $X_0$ is irreducible and the smooth fibers $X_{t\neq 0}$ are $K$-trivial, we have $K_{\cx}=0$ with nowhere-vanishing section $\Omega$.   Writing $\omega_0 :=\mathrm{Res}(\tfrac{\Omega}{f})$ and $\eta\colon \tilde{X}_0\to X_0$ for the normalization, $\eta^*\omega_0$ is a section of $\Omega_{\tilde{X}_0}^2\langle\log\,E\rangle$ with nonvanishing residue $\omega_E \in \Omega^1(E)$.  Therefore $H^2(X_0\setminus \bar{\mathsf{S}}_1)\cong H^2(\tilde{X}_0\setminus E)$ has a nontrivial $(1,2)+(2,1)$ part, which must come from $H_2(X_0)(-2)$ in the localization sequence since $H_1(\bar{\mathsf{S}}_1)=\{0\}$.\footnote{The localization sequence is a bit subtle in the singular case:  we need the setting where $X$ is a singular surface containing a curve $Y$, with $X\setminus Y$ smooth.  Then the homology sequence $\to H_2(X)\to H_2(X,Y)\to H_1(Y)$ twisted by $\QQ(-2)$ is $\to H_2(X)(-2)\to H^2(X\setminus Y)\to H_1(Y)(-2)$.} That is, $H^2(X_0)$ has a $(1,0)+(0,1)$ part, as claimed.

The main point is of course that we can say all this \emph{without} resolving $X_0$, let alone performing semistable reduction (or similar).  This will be refined by Example \ref{ex-degk3} below, in the setting of degenerating quartics in $\PP^3$.
\end{example}

\section{The $J_{k,\infty}$ series}\label{S7.4}

We continue our study of the effect of non-isolated surface singularities on the Hodge theory of degenerations, by considering some non-slc examples. Namely, 
one of the simplest series of \emph{non}-slc non-isolated surface singularities are the 
$$J_{\kappa,\infty}:\;\;\;\;f_{\kappa}\;\underset{\text{loc}}{\sim} \; x^2 + y^3 + y^2 z^{\kappa}$$
for $\kappa\geq 3$.  (Note that $J_{1,\infty}$ resp. $J_{2,\infty}$ are the singularities $D_{\infty}$ and $T_{2,3,\infty}$ considered above; however, the formulas derived below apply to them too.) Our interest in this class of singularities is partially motivated by their occurrence in the study of projective degenerations of $K3$ surfaces (e.g. see \cite{log2}; see also \cite[\S3]{KL} for further related discussion in the isolated singularity case). In this context, we note that an interesting byproduct of our study (Theorem \ref{th7.4} and Remark \ref{rem-jk} below) is a conceptual explanation of the fact (observed heuristically for quartics in \cite{log2}) that that the worst $J_{\kappa,\infty}$ that can occur in a degeneration of $K3$ surfaces is $J_{4,\infty}$.

 Locally, for $J_{\kappa,\infty}$,   the singular locus $Z$ is just the $z$-disk $\{x=y=0\}$, and taking $\fg=z$ gives $\Lambda^{\circ}=\{t=\tfrac{4}{27}s^{3\kappa}\}$ so that $\mathfrak{r}=3\kappa$. Away from $\uo$ (on $Z^*$) we have $A_{\infty}$ as above, with the branches exchanged about $\uo$ (i.e. $\beta=\tfrac{1}{2}$) $\iff$ $\kappa$ is odd. By connectedness of $\mathfrak{F}_{y^3+y^2 z^{\kappa}}$ and \eqref{eq-ST*}, $\sigma^1_{f_{\kappa},\uo}=0$. Applying \eqref{eqsss2} with $r=3k$ (without the tildes) and taking the $\mu$-constant deformation $f+z^{3\kappa} \rightsquigarrow x^2 + y^3 + z^{3\kappa}$ gives
\begin{equation}\label{eq-sk*}
\begin{split}
\sigma_{f_{\kappa},\uo}^2 &= \sigma_{x^2 + y^3 + z^{3\kappa}}-\sum_{k=0}^{3\kappa-1} [1+\tfrac{\beta+k}{3\kappa}]\\
&= 	\{\tfrac{5}{6},\tfrac{7}{6}\}*\{\tfrac{1}{3\kappa},\tfrac{2}{3\kappa},\ldots,\tfrac{3\kappa-1}{3\kappa}\}-\{1\text{ resp.~}\tfrac{1+6\kappa}{6\kappa}\}*\{0,\tfrac{1}{3\kappa},\tfrac{2}{3\kappa},\ldots,\tfrac{3\kappa-1}{3\kappa}\}
\end{split}
\end{equation}
by $\S$\textbf{II.2.1} (where $\{A\}*\{B\}:=\sum_{\alpha\in A,\beta\in B}[\alpha+\beta]$). One easily sees that all the negative terms are cancelled, and that the only remaining integer term is $[2]$ is $\kappa$ is even (and nothing if $\kappa$ is odd). To show that the accompanying weight is $4$ (whereas all the others are $2$), we must use \eqref{eqsss2} with $r=3\kappa+1$; rather than doing this in full, one just needs that $f+z^{3\kappa+1}$ has $h^{2,2}_{\van,\uo}=1$ for $\kappa$ even (and $0$ for $\kappa$ odd). This follows at once from Theorem \textbf{II.5.4}, allowing us to conclude for $\kappa$ odd
\begin{equation}\label{eq-sk2*}
\tilde{\sigma}_{f_{\kappa},\uo}^2 = \sum_{m=1}^{\frac{\kappa-1}{2}}\{[(\tfrac{5\kappa+2m}{6\kappa},2)]+[(\tfrac{13\kappa-2m}{6\kappa},2)]\}+\sum_{m'=1}^{2\kappa-1}[(\tfrac{7\kappa+2m}{6\kappa},2)]
\end{equation}
resp. for $\kappa$ even
\begin{equation}\label{eq-sk3*}
\tilde{\sigma}_{f_{\kappa},\uo}^2 = \sum_{m=1}^{\frac{\kappa}{2}-1}\{[(\tfrac{\frac{5\kappa}{2}+m}{3\kappa},2)]+[(\tfrac{\frac{13\kappa}{2}-m}{3\kappa},2)]\}+\sum_{m'=1}^{2\kappa-1}[(\tfrac{\frac{7\kappa}{2}+m}{3\kappa},2)]+[(2,4)].	
\end{equation}
In particular, $T^{\text{ss}}$ has order $6\kappa$ resp. $3\kappa$, and $N$ is trivial on $H^2(\mathfrak{F}_{f_{\kappa}})$.

\begin{example}\label{ex7(3)}
Continuing where Example \ref{ex7(2)} left off, we can look at various ``degenerations'' of the singular configuration there by colliding pinch points:  from (\emph{i}) four $J_{1,\infty}=D_{\infty}$'s (Ex. \ref{ex7(2)}), to (\emph{ii}) two $J_{1,\infty}$'s and one $J_{2,\infty}=T_{2,3,\infty}$, to (\emph{iii}) one $J_{1,\infty}$ and one $J_{3,\infty}$ or (\emph{iii}$'$) two $J_{2,\infty}$'s, to (\emph{iv}) one $J_{4,\infty}$, all on a smooth $\PP^1$ of $A_{\infty}$'s.  For $p$ of type $J_{\kappa,\infty}$, the MHS on $V_{f_{\kappa},\uo}^2$ takes the form
\[
\includegraphics[scale=0.5]{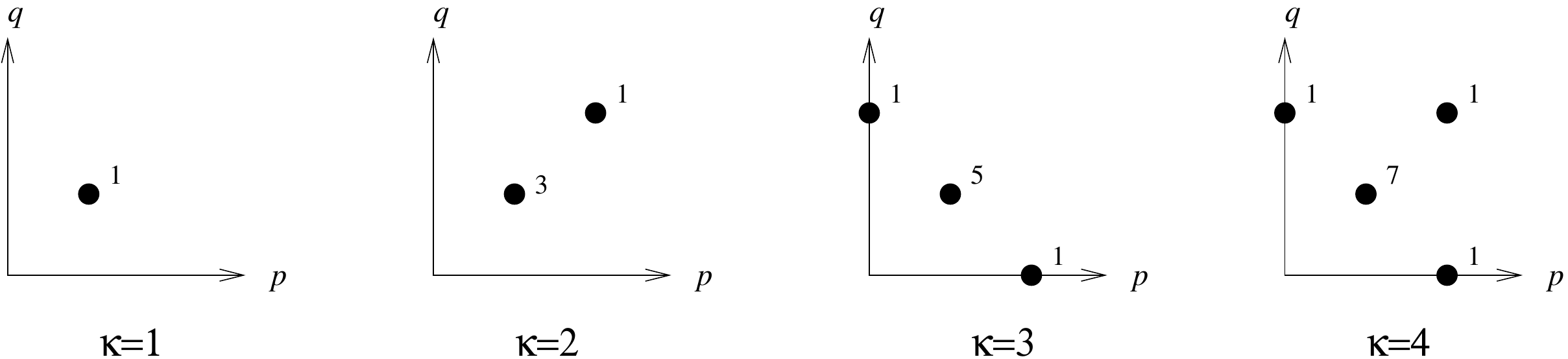}
\]
where only the $(2,2)$ classes are $T^{\text{ss}}$-invariant (and $T^{\text{ss}}$ has order $18$ resp. $12$ on the $(2,0)$ part for $\kappa=3$ resp. $4$).

To compute $H^2_{\van}(X_t)$, we use \eqref{eq-SS*} and $\ch^1\phi_f\QQ_{\cx}\cong \jmath_! \cl(-1)$, where $\cl$ has monodromy (of $-1$) about only the $J_{1,\infty}$ and $J_{3,\infty}$ points. Moreover, the fact that the general fibers of $f\colon \cx\to \Delta$ are $K3$ implies that for (\emph{iii}$'$), (\emph{iv}) a $(2,2)$ class in $E_2^{0,2}$ must cancel with $E_2^{2,1}$. (Otherwise we would have $\text{rk}\,\gr_F^2 H^2_{\van}(X_t)=2$.) This yields for $H^2_{\van}$
\[
\includegraphics[scale=0.5]{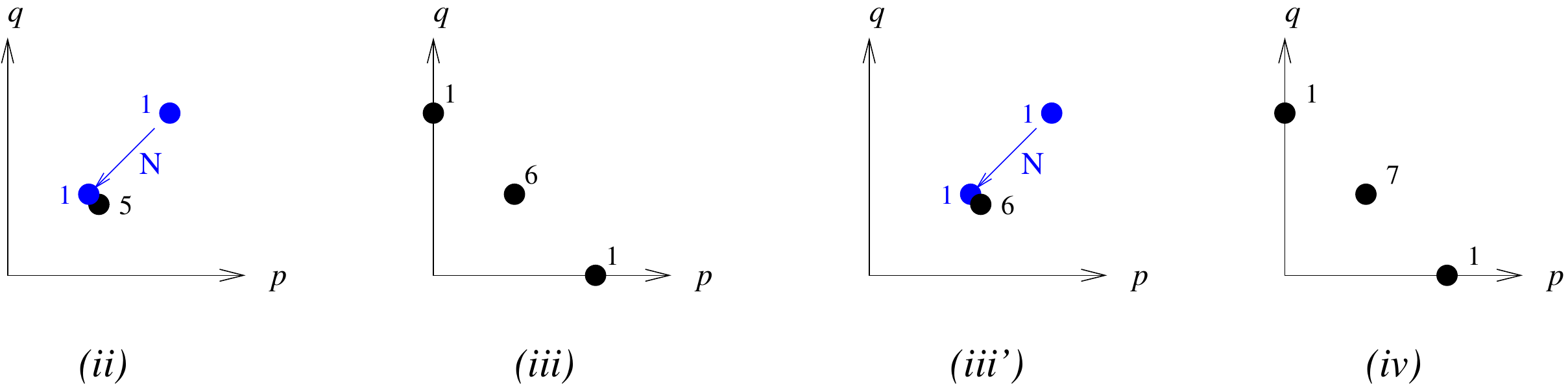} 
\]
so that (\emph{ii}), (\emph{iii}$'$) [resp. (\emph{iii}), (\emph{iv})] are degenerations of type III [resp. I].  Only the classes in blue are $T^{\text{ss}}$-invariant.
\end{example}

Here is an intriguing consequence of \eqref{eq-sk2*}-\eqref{eq-sk3*}, which will be further strengthened in the next subsection (Cor. \ref{cor-dis}):
\begin{thm}\label{th7.4}
Let $\cx\overset{f}{\to}\Delta$ be a degeneration of surfaces with smooth total space, generic geometric genus $p_g=h^{2,0}(X_t)$, reduced special  fiber $X_0$, and $p\in X_0$ of type $J_{\kappa,\infty}$. Then $p_g \geq \lfloor \tfrac{\kappa-1}{2}\rfloor$.
\end{thm}

\begin{proof}
The formulas show that $h^{2,0}(V_{f,p}^2)=\lfloor \tfrac{\kappa-1}{2}\rfloor$ and $h^{2,2}(V_{f,p}^2)=0$ resp. $1$ (for $\kappa$ odd resp. even) for the summand $V_{f,p}^2$ contributed by $p$ to $E_2^{0,2}$. On the other hand, $E_2^{2,1}=H^2(\ch^1\phi_f\QQ_{\cx})$ can only have Hodge type $(2,2)$, so cannot cancel the $(2,0)$ part of $E_2^{0,2}$ under $d_2$. Hence $h^{2,0}(X_t)\geq h^{2,0}(H_{\lm}^2)\geq h^{2,0}(H_{\van}^2)\geq \lfloor \tfrac{\kappa-1}{2}\rfloor$.	
\end{proof}

\begin{rem} \label{rem-jk}
Of course, the same proof shows that if $\{p_j\}\subset \text{sing}(X_0)$ are of type $J_{\kappa_j,\infty}$, then $p_g(X_t)\geq \sum_j \lfloor \tfrac{\kappa_j -1}{2}\rfloor$.	 In particular, when $p_g(X_t)=1$ (e.g. $K3$ surfaces), the only non-slc singularities of type $J_{\kappa,\infty}$ that can occur are $J_{3,\infty}$ and $J_{4,\infty}$, and at most one such singularity can occur. This matches with the detailed analysis of the degenerations of quartic $K3$ surfaces of \cite{log2}. Of course, the point is that Theorem \ref{th7.4} recovers this in much more generality and more conceptually. 
\end{rem}

\section{Clemens-Schmid discrepancies for nodal total spaces}\label{S7.5}

One of the new phenomena that occurs in the presence of non-isolated singularities for $X_0$ is the fact that the total space $\cx$ is typically singular. In this section we comment on this situation, under the genericity assumption that $\cx$ has only ordinary double points. This occurs for instance when considering generic smoothings of hypersurfaces $X_0$ with $1$-dimensional singular locus (and $A_\infty$ generic singularities). These types of examples occur frequently in the study of compactifications of $K3$ surfaces or surfaces of general type.

In general, when $\cx$ is singular along $X_0$, the Clemens-Schmid sequence and its two main corollaries
\begin{enumerate}[label=\bf(\Alph*),leftmargin=1cm]
\item $H^k(X_0)$ surjects onto the $T$-invariants $H^k_{\lm}(X_t)^T$ and
\item $H^{k+1}_{\pha}(X_0)=\text{im}\{H^k_{\van}(X_t)\to H^{k+1}(X_0)\}$ is pure of weight $k+1$
\end{enumerate}
fail.  Various approaches to quantifying or bounding this failure were described in \S\textbf{I.8-9}.  For the non-isolated singularities considered in this section, with $\dim(\text{sing}(X_0))=1$, most often one has type $A_{\infty}$ (transverse $A_1$) on the ``generic stratum'' $\mathsf{S}_1 \subseteq \text{sing}(X_0)$.  When $X_0 = \{F(\ux)=0\}$ is a hypersurface of degree $d$, and $\cx=\{ F(\ux)+tG(\ux)=0\mid t\in\Delta\}$ a generic smoothing (with $\deg(G)=d$), $\text{sing}(\cx)$ consists of $d\cdot \deg(\text{sing}(X_0))$ nodes contained in $\mathsf{S}_1$. In this case we can say much more about (A) and (B).  This is crucial, for instance, in the Hodge-theoretic analysis of how the $\{J_{k,\infty}\}$ singularities actually arise in GIT.

\begin{thm}\label{th-csd}
Suppose the total space of $f\colon \cx \to \Delta$ has $\rka$ nodes on $X_0$, and no other singularities. (The singular fiber itself may have arbitrary singularities and nonreduced components.) Recall that $n+1=\dim(\cx)$.
\begin{enumerate}[label=\textup{({\roman*})},leftmargin=1cm]
\item If $n$ is odd, then \textup{(A)} and \textup{(B)} hold.
\item If $n=2m$ is even, \textup{(A)} and \textup{(B)} can only fail when $k=2m$, and then only in type $(m,m)$.  Put
 \begin{equation*}
   \begin{split}
      \ra &:= \dim(\textup{coker}\{H^{2m}(X_0)^{m,m}\overset{\sp^{2m}}{\to} (H^{2m}_{\lm}(X_t)^T)^{m,m}\} ) \;\;\;\;\;\text{and}\\
      \rb &:= \dim\left(H_{\pha}^{2m+1}(X_0)^{m,m}\right) = \textup{rank}\{(H^{2m}_{\van}(X_t))^{m,m}\overset{\delta}{\to}H^{2m+1}(X_0)\}
   \end{split}
 \end{equation*}
for the corresponding ``Clemens--Schmid discrepancies''; then we have $\ra+\rb\leq\rka$.  In particular, if $m=1$ and $X_0$ is irreducible, then $\ra+\rb=\rka$.
\end{enumerate}
\end{thm}

\begin{rem}
Taking $T$-invariants of type $(m,m)$ in \eqref{eq7.0} yields
\begin{equation}\label{eq-csd}
0\to \frac{[H^{2m}_{\lm}(X_t)^{m,m}]^T}{\sp(H^{2m}(X_0)^{m,m})}\to \frac{[H^{2m}_{\van}(X_t)^{m,m}]^T}{\mathrm{im}(N)}\to H^{2m+1}_{\pha}(X_0)^{m,m}\to 0\,,
\end{equation}
in which the end terms have ranks $\ra$ and $\rb$. So one way to view the Theorem is as saying that $\rka$ is an upper bound on the middle term (which would be zero were $\cx$ smooth).
\end{rem}

\begin{proof}[Proof of Theorem]
Consider a resolution $\pi\colon \tilde{\cx}\to \cx$ with exceptional quadrics $Q_i :=\pi^{-1}(q_i)$ ($i=1,\ldots,\rka$) and associated long-exact sequence
\begin{equation}\label{eqp-1}
\cdots \to H^k(\cx)\to H^k(\tilde{\cx})\to \oplus_{i=1}^{\rka}\tilde{H}^k(Q_i)\to \cdots
\end{equation}
of MHS.\footnote{As usual, $H^*(\cx)$, $H^*(\tilde{\cx})$, $H^*_c(\tilde{\cx})$, $\IH^*(\cx)$, $\IH^*_c(\cx)$ acquire their MHSs via identification with $H^*(X_0)$, $H^*(\pi^{-1}(X_0))$, $H_{2n+2-*}(\pi^{-1}(X_0))$, $\HH^*(\imath_{X_0}^*\IC_{\cx}[-n-1])$, and $\HH^*(\imath^!_{X_0}\IC_{\cx}[-n-1])$.} Applying the Decomposition Theorem to $\pi$ (cf.~\textbf{(I.7.12)}) gives
$$ R\pi_*\QQ_{\tilde{\cx}}[n+1]\simeq \IC_{\cx}\oplus \bigoplus_{i,j} \imath^{q_i}_* \left\{ \begin{matrix} H^{n+1+j}(Q_i)[-j], & j\geq 0 \\ H_{n+1-j}(Q_i)(-n-1)[-j], & j<0 \end{matrix}\right.$$
hence
\begin{equation}\label{eqp-2}
H^k(\tilde{\cx})\cong \IH^k(\cx)\oplus \bigoplus_i \left\{ \begin{matrix} H^k(Q_i),& k>n\\ H^{k-2}(Q_i)(-1),& k\leq n\end{matrix}\right.
\end{equation}
and
\begin{equation}\label{eqp-3}
H_c^k(\tilde{\cx})\cong \IH_c^k(\cx)\oplus \bigoplus_i \left\{ \begin{matrix} H^k(Q_i),& k>n\\ H^{k-2}(Q_i)(-1),& k\leq n\end{matrix}\right. ,
\end{equation}
while perversity of $\QQ_{\cx}[n+1]$ yields the s.e.s.
\begin{equation}\label{eqp-4}
0\to \oplus_i \imath^{q_i}_* V_i \to \QQ_{\cx}[n+1] \to \IC_{\cx}\to 0
\end{equation}
in $\mathrm{MHM}(\cx)$ (for some MHSs $\{V_i\}$).  From \eqref{eqp-1}-\eqref{eqp-2} we deduce that $V_i\cong H^n_{\text{pr}}(Q_i)$, which is zero for $n$ odd, making $\cx$ an intersection homology manifold; in this case the results of \S\textbf{I.5} for $\cx$ smooth (including Clemens-Schmid) go through.

When $n=2m$ is even, one has $V_i\cong \QQ(-m)$, and \eqref{eqp-2}-\eqref{eqp-3} imply that the weights of $\IH^k(\tilde{\cx})$ resp. $\IH^k_c(\tilde{\cx})$ are $\leq k$ resp. $\geq k$.  This from the $\IH$ C-S sequence \textbf{(I.5.7)}
\begin{equation}\label{eqp-5}
\xymatrix{\IH^k_c(\cx) \ar [r] & \IH^k(\cx) \ar [r]^{\sp_{\IH}^k} & H^k_{\lm}(X_t)^T \ar [r] & 0 \\ & H^k(\cx) \ar [u]^{\sigma^k} \ar [ur]^{\sp^k}}
\end{equation}
we get that $\ker(\sp^k_{\IH})$ has pure weight $k$.  By \eqref{eqp-4}, \small
\begin{equation}\label{eqp-6}
0\to H^{2m}(\cx)\overset{\sigma^{2m}}{\to}\IH^{2m}(\cx)\overset{\alpha}{\to}\QQ(-m)^{\oplus \rka}\overset{\beta}{\to}H^{2m+1}(\cx)\overset{\sigma^{2m+1}}{\to}\IH^{2m+1}(\cx)\to 0
\end{equation}\normalsize
is exact, and $\sigma^k$ is an isomorphism for $k\neq 2m,2m+1$. Conclude that $\mathrm{coker}(\sp^k)=\{0\}$ for $k\neq 2m$, $\ker(\sp^k)$ has weight $k$ for $k\neq 2m+1$, while $\mathrm{coker}(\sp^{2m})\hookrightarrow \QQ(-m)^{\oplus\mathrm{rk}(\alpha)}$ ($\implies \ra\leq\mathrm{rk}(\alpha)$) and
\begin{equation}\label{eqp-7}
0\to \QQ(-m)^{\oplus \mathrm{rk}(\beta)} \to \ker(\sp^{2m+1})\to \ker (\sp_{\IH}^{2m+1})\to 0
\end{equation}
is exact ($\implies \rb=\mathrm{rk}(\beta)$).

Finally, if $m=1$ and $\rg$ [resp. $\rg+\rka$] is the number of components of $X_0$ [resp. $\pi^{-1}(X_0)$], then by \eqref{eqp-2}-\eqref{eqp-3} and C-S for $\tilde{\cx}$,
\begin{equation*}
\begin{split}
\mathrm{rk}(\IH^2_c(\cx)\to \IH^2(\cx))&= \mathrm{rk}(H^2_c(\tilde{\cx})\to H^2(\tilde{\cx}))-\rka\\
&=\mathrm{cork}\{H^0_{\lm}(X_t)\to H_4(\pi^{-1}(X_0))(-2)\}-\rka\\
&=(\rg+\rka-1)-\rka = \rg-1.
\end{split}
\end{equation*}
So when $\rg=1$, $\sp^2_{\IH}$ is am isomorphism and $\ra=\mathrm{rk}(\alpha)$.
\end{proof}

In order to make use of this Theorem, we need a complementary result on how to modify the computations of $H_{\van}^*$ in previous sections.

\begin{prop}\label{prop-dis}
Let $Z$ be a component of $\mathrm{sing}(X_0)$, $Z_1=Z\cap \mathsf{S}_1$ the $A_{\infty}$ locus, and $Z_1^*\underset{\jmath}{\hookrightarrow}Z_1$ the complement of the nodes of $\cx$. Then $\ch^j\phi_f\QQ_{\cx}|_{Z_1}=0$ for $j\neq n-1$, and $\ch^{n-1}\phi_f\QQ_{\cx}|_{Z_1}\cong\jmath_! \mathcal{L}(-\lfloor\tfrac{n}{2}\rfloor)_{(-1)^n}$, where the subscript indicates the action of $T^{\mathrm{ss}}$, and $\mathcal{L}$ is a local system \textup{(}pointwise $\cong \QQ(0)$\textup{)} with monodromy $(-1)^n$ about each point of $Z_1\setminus Z_1^*$. \textup{(}Its monodromy about $Z\setminus Z_1$ depends on the singularity types there.\textup{)}
\end{prop}

\begin{proof}
We may choose local coordinates $(z_1,\ldots,z_{n+1})$ in which $Z_1 =\{z_1=\cdots =z_n=0\}$ and $q\in Z_1\setminus Z_1^*$ is $\uo$, and $F+tG=0$ is $\sum_{\ell=1}^n z_{\ell}^2 = tz_{n+1}$.  Hence the Milnor fiber at $q$ is contractible, and $\tilde{H}^k(\mathfrak{F}_{f,q})=0$ ($\forall k$). It is also clear from the equation that the ``vertical'' monodromy (in $z_{n+1}$) equals the ``horizontal'' monodromy (in $t$) on $\tilde{H}^{n-1}(\mathfrak{F}_{f,p})$ for $p\in Z_1^*$ near $q$.
\end{proof}

\begin{example}\label{ex-degk3}
Let $\cx$ be a degeneration of quartic $K3$ surfaces, with $X_0$ as in Examples \ref{ex7(2)}-\ref{ex7(3)}, $\mathrm{sing}(X_0)=\bar{\mathsf{S}}_1=Z$ a smooth conic curve, and $\mathrm{sing}(\cx)=$ 8 points on $\mathsf{S}_1=Z_1$.  In the computation of $H^2_{\van}(X_t)$, only $H^1(\ch^1\phi_f\QQ_{\cx})$ ($=H^1(\bar{\mathsf{S}}_1,\jmath'_!\mathcal{L}(-1))$, where $\jmath'\colon Z_1^*\hookrightarrow Z$) changes.  In each of the five cases (\emph{i})-(\emph{iv}), the effect is simply to add $8$ $T$-invariant $(1,1)$-classes to $H^2_{\van}(X_t)$.  So the middle term of \eqref{eq-csd} (with $m=1$) has rank $8$; indeed, by Theorem \ref{th-csd} we must have $\ra+\rb=8$.

This determines the MHS types for $H^2_{\lm}$ and $H^2_{\van}$ in each case,\footnote{It should be noted that in case (\emph{i}) the LMHS still must have type II. Blowing up $\cx$ along the base locus $F=G=0$ in $\PP^3$ gives a crepant resolution $\hat{\cx}$ (with exceptional $\PP^1$'s $\{L_i\}$) to which we may apply the argument from Example \ref{ex7(2)}. The exact sequence $(\{0\}=)\oplus H^1(L_i)\to H^2(\hat{X}_0)\to H^2(X_0)$ then passes the nontrivial weight-1 part of $H^2(X_0)$ to $H^2(\hat{X}_0)$ hence to $H^2_{\lm}$.} while leaving an apparent ambiguity in $H^2(X_0)$:
\[\begin{tikzpicture}[scale=0.85]
\foreach \j in {0,...,4} {\draw [thick,gray,<->] (\j*3,2) -- (\j*3,0) -- (2+\j*3,0);}
\foreach \j in {0,...,4} {\node at (\j*3+2.1,0) {\tiny $p$};}
\foreach \j in {0,...,4} {\node at (0.1+\j*3,2.1) {\tiny $q$};}
\node at (1,-0.6) {$(i)$};
\node at (4,-0.6) {$(ii)$};
\node at (7,-0.6) {$(iii)$};
\node at (10,-0.6) {$(iii')$};
\node at (13,-0.6) {$(iv)$};
\foreach \j in {0,...,4} {\filldraw (1+3*\j,1) circle (3pt);}
\foreach \j in {0,1,2} {\node at (1.5+3*\j,1.2) {\tiny $6+\rb$};}
\foreach \j in {3,4} {\node at (1.5+3*\j,1.2) {\tiny $5+\rb$};}
\foreach \j in {1,3} {\filldraw (3*\j,0) circle (3pt);}
\foreach \j in {1,3} {\node at (3*\j+0.2,0.2) {\tiny $1$};}
\foreach \j in {0,1} {\filldraw (0+\j,1-\j) circle (3pt);}
\foreach \j in {0,1} {\node at (0.2+\j,1.2-\j) {\tiny $1$};}
\end{tikzpicture}\]
We claim that, in fact, $\rb=0$ in all cases.  Let $\hat{X}_0\overset{\eta}{\twoheadrightarrow}X_0$ be the normalization; this is a smooth $\mathrm{dP}_4$, or a $\mathrm{dP}_4$ with one node which is \emph{not} on $\hat{Z}:=\eta^{-1}(Z)$. Here $\hat{Z}\twoheadrightarrow Z$ is the double cover branched at the $J_{k,\infty}$ points:  $\hat{Z}$ is (\emph{i}) smooth elliptic; (\emph{ii}) nodal rational; (\emph{iii}) cuspidal rational; (\emph{iii}$'$) 2 rational curves meeting in a pair of nodes; or (\emph{iv}) 2 rational curves meeting in a tacnode.  Via the exact sequence $H^2(Z)\oplus H^2(\hat{X}_0)\to H^2(\hat{Z})\to H^3(X_0)\to H^3(\hat{X}_0)$ we therefore have that $H^3(X_0)=\{0\}$, hence in particular that $H^3_{\pha}(X_0)^{1,1}=\{0\}$, as desired.\footnote{The main point is to show that, in cases (\emph{iii}$'$) and (\emph{iv}), the two components of $\hat{Z}$ have intersection matrix $\tiny \begin{pmatrix}0&2\\2&0\end{pmatrix}$, so that $H^2(\hat{X}_0)$ surjects onto $H^2(\hat{Z})$.}
\end{example}

\begin{cor}\label{cor-dis}
Theorem \ref{th7.4} and Remark \ref{rem-jk} remain true for degenerations of surfaces whose total spaces admit nodes on the $A_{\infty}$ locus of $\mathrm{sing}(X_0)$.
\end{cor}

\begin{proof}
By Theorem \ref{th-csd}, we still have $h^{2,0}(H^2_{\lm})\geq h^{2,0}(H^2_{\van})$; and Prop. \ref{prop-dis} makes clear that $H^2(\ch^1\phi_f\QQ_{\cx})$ still has no $(2,0)$ part.
\end{proof}

\section{The double box}\label{S7.6}

As a final application, we discuss a higher dimensional example, where in particular the notion of higher rational singularities become relevant. Specifically, we determine what these tools can say about the cohomology of the cubic hypersurface $X_0\subset \PP^6$ cut out by the second Symanzik polynomial associated to the ``double box'' Feynman diagram with general masses and momenta (in physics ``dimension'' $D\geq 4$).  For $D=4$, Bloch \cite{Bl21} showed that $\gr^W_5 H^5(X_0)\cong H^1(E)(-2)$ for some elliptic curve $E$.  Subsequently, Doran, Harder and Vanhove \cite{DHV23} determined $E$ and discovered that when $D>4$ it is replaced by a curve of genus $2$. Our goal here is to compute the rest of $H^5(X_0)$ and to say more about the monodromy of a smoothing of $X_0$.  From our point of view, this is interesting as a nontrivial example of a degeneration with non-isolated 1-rational (hence also 1-log-canonical) singularities, cf.~Rem.~\ref{rem-1rat}.  We will also discuss the degeneration that arises from specializing momenta \emph{en route} from $D>4$ to $D=4$.

\begin{rem}\label{rem-1rat}
Using local Bernstein-Sato polynomials, one easily computes that the microlocal log-canonical threshold $\tilde{\alpha}_{X_0}=\frac{5}{2}$; that $X_0$ is 1-rational then follows from the appendix to \cite{FL-DuBois}.  This is not used in the remainder of the section, though the calculations do provide a nice illustration of \eqref{eq-th1.6} together with the additional isomorphism $\gr^2_F H^*(X_0)\cong \gr^2_F(H^*_{\lim}(X_t))^{T_{\text{ss}}}$ that comes from singularities being 1-rational.
\end{rem}

Thus we shall consider a degeneration $\cx \to \Delta$ of the form $0=F(\UZ)+tG(\UZ)$, where $G$ is a general cubic, and the $2^{\text{nd}}$ Symanzik polynomial $F$ takes the form
$$F(\UZ)=Z_{0123}Q(Z_4,Z_5,Z_6)+Z_{3456}Q'(Z_0,Z_1,Z_2)+Z_3 P(\UZ),$$
with $P(\UZ)=\sum_{i=0}^3\sum_{j=3}^6 P_{ij}Z_iZ_j$ (and $P_{33}=0$) and $Q,Q'$ quadrics, and where we denote $Z_0+\cdots+Z_3=:Z_{0123}$ etc.~for convenience. More precisely, let $\mo_1,\ldots,\mo_6\in \RR^{1,D-1}$ denote ``momenta'' summing to $\vec{0}$ and $\ms_0,\ldots,\ms_6\in \RR$ denote ``masses'', \emph{all assumed generic}; then writing
$$U(\UZ)=Z_{012}Z_{456}+Z_3Z_{012456}$$
for the $1^{\text{st}}$ Symanzik polynomial, we have 
\begin{align*}
Q'(\UZ)&=\mo_2^2 Z_0Z_1+\mo_{23}^2 Z_0Z_2+\mo_3^2Z_1Z_2+Z_{012}\textstyle\sum_{i=0}^2 \ms_i^2 Z_i,\\
Q(\UZ)&=\mo_6^2Z_4Z_5+\mo_{56}^2Z_4Z_6+\mo_5^2Z_5Z_6+Z_{456}\textstyle\sum_{i=4}^6 \ms_i^2 Z_i,
\end{align*}
and
\begin{align*}
\textstyle P(\UZ)=\sum_{i=0}^2\sum_{j=4}^6\left(\sum_{k=i+2}^{10-j}\mo_k\right)^2Z_iZ_j&\textstyle +Z_{012}\sum_{j=4}^6\ms_j^2Z_j\\ &\textstyle +Z_{456}\sum_{i=0}^2\ms_i^2Z_i+\ms_3^2U(\UZ).
\end{align*}
If $D>4$, then $\text{sing}(X_0)$ consists of the two disjoint conic curves 
\begin{align*}
C&=\{Z_0=Z_1=Z_2=Z_3=Q=0\}\;\;\;\text{and}\\
C'&=\{Q'=Z_3=Z_4=Z_5=Z_6=0\}.
\end{align*}
When $D=4$, $\text{sing}(X_0)=C\disj C'\disj \cs''$, where $\cs''$ consists of two nodes $p_1'',p_2''$ with nonzero $Z_3$-coordinates.\footnote{This is contrary to the claim in \cite[Prop.~2.1]{Bl21}. The equations of these nodes are quite ugly: write $\mo_4=\a_2\mo_2+\a_3\mo_3+\a_5\mo_5+\a_6\mo_6$ for the linear dependency forced by $D=4$ (of course also $\mo_1=-\mo_2-\mo_3-\mo_4-\mo_5-\mo_6$). Define quantities $\N'=\a_2(1+\a_2)\mo_2^2+2\a_2(1+\a_3)\mo_2\cdot\mo_3+\a_3(1+\a_3)\mo_3^2$, $\N=\a_5(1+\a_5)\mo_5^2+2\a_6(1+\a_5)\mo_5\cdot\mo_6+\a_6(1+\a_6)\mo_6^2$, $\M'=-\a_2\ms_0^2+(\a_2-\a_3)\ms_1^2+(1+\a_2)\ms_2^2$, and $\M=-\a_5\ms_4^2+(\a_6-\a_5)\ms_5^2+(1+\a_5)\ms_6^2$.  Let $\rho_1,\rho_2$ be the roots of the quadratic equation $$(\N'-\M')\rho^2+(\N'-\M'+\N-\M+\ms_3^2)\rho+(\N-\M)=0$$ (which arises from intersecting $U=0$ with five hyperplanes). Then $$p_i''=[-\a_2\rho_i:(\a_2-\a_3)\rho_i:(1+\a_3)\rho_i:\tfrac{(\N'-\M')\rho_i+(\N-\M)}{\ms_3^2}:-\a_6:(\a_6-\a_5):1+\a_5].$$} It will be important later that $U(p_i'')=0$.

The exceptional divisors of the blowup along $C\disj C'$ are quadric bundles with smooth total spaces \cite[Lem.~3.1]{Bl21}; in particular, their singular fibers have corank 1. Since the latter are determined by the intersections of $C,C'$ with a cubic discriminant locus, there are 6 on each bundle. Equivalently, $X_0$ has generically $A_{\infty}$ singularities on $C\disj C'$, with pinch points at $\cs=\{p_1,\ldots,p_6\}\subset C$ and $\cs'=\{p_1',\ldots,p_6'\}\subset C'$; the stratification of $\text{sing}(X_0)$ is thus given by $\str_1=C\stm \cs\disj C'\stm \cs'$ and $\str_0=\cs\disj \cs'\disj \cs''$.  Write $\rd:=|\cs''|=0$ ($D>4$) resp.~$2$ ($D=4$), so that $|\str_0|=12+\rd$.

We must also keep track of the $12$ points $\{q_1,\ldots,q_6\}\subset C$ and $\{q_1',\ldots,q_6'\}\subset C'$ given by intersecting these curves with $G=0$, which are the (nodal) singularities of the total space. Denote by $\str_1^*$ the complement of these points in $\str_1$. As $n=5$, we are in the milder case (i) of Theorem \ref{th-csd}, and they do not affect Clemens-Schmid.  Writing $\ch^k_{\van}:=\ch^k\phi_t\QQ_{\cx}$, we know that $\ch^k_{\van}=0$ for $k\neq 4,5$.

To evaluate $\ch^4_{\van}$ and $\ch^5_{\van}$ (stalkwise), we consider the local forms of $\cx$ at points of $\text{sing}(X_0)$:
\begin{itemize}[leftmargin=0.5cm]
\item at $p\in \cs''$ ($D=4$ only):  $\sum_{\ell=1}^6 x_{\ell}^2=t$ $\implies$ $\ch^5_{\van}\simeq \QQ(-3)$ (and $\ch^4=0$)
\item at $p\in \str_1^*\subset C\disj C'$: $\sum_{\ell=1}^5x_{\ell}^2=t$ ($x_6$ free) $\implies$ $\ch^4_{\van}\simeq \QQ(-2)_-$ (and $\ch_{\van}^5=0$), where the subscript ``$-$'' means that $T_{\text{ss}}$ acts by $-1$ (Prop.~\ref{prop-dis});
\item at $q_i,q_i'$: $\sum_{\ell=1}^5 x_{\ell}^2=tx_6$ $\implies$ $\ch^5_{\van}=\ch^4_{\van}=0$, and the generic $\ch_{\van}^4$ on $\str_1^*$ has vertical monodromy $-1$ as $p$ goes about these points (Prop.~\ref{prop-dis} again); and
\item at $p_i,p_i'$: $\sum_{\ell=1}^4x_{\ell}^2+x_6x_5^2=t$ $\implies$ $\ch_{\van}^5\simeq \QQ(-3)$ (and $\ch^4_{\van}=0$), and the generic $\ch^4_{\van}$ on $\str_1^*$ has vertical monodromy $-1$ as $p$ goes about these points.  (This is by repeating the analysis of pinch points in \S\ref{S7.3}, which yields $\tilde{\sigma}_{F,p_i}^4=0$ and $\tilde{\sigma}_{F,p_i}^5=[(3,6)]$.)
\end{itemize}
Thus $\ch^5\phi_t\QQ_{\cx}$ is a sum of $12+\rd$ skyscraper $\QQ(-3)$'s on $\str_0$; while $\ch^4\phi_t\QQ_{\cx}\simeq\kappa_!(\QQ(-2)_-\otimes \chi)$, where $\kappa\colon \str_1^*\hookrightarrow C\disj C'$, and $\chi$ is the rank-1 ``Kummer sheaf'' with monodromies of $-1$ about each of the points $p_i,p_i',q_i,q_i'$. Since there are $12$ such points on each conic curve, the double-covers $\tilde{C}$, $\tilde{C}'$ (of $C,C'$) branched along them are curves of of genus 5.  Writing $H:=H^1(\tilde{C})\oplus H^1(\tilde{C}')$, the spectral sequence \eqref{eq-SS*} becomes 
$$\xymatrix@C=1em@R=1em{\ch^5 & \QQ(-3)^{\oplus 12+\rd}  & 0 & 0 \\ \ch^4 & 0 & V(-2)_- & 0  \\ & H^0 & H^1 & H^2 \ar@{-}(8,-15);(70,-15)
\ar@{-}(8,-15);(8,6)} $$	
whence $H^5_{\van}(X_t)\cong V(-2)_-\oplus \QQ(-3)^{\oplus 12+\rd}$ (the subscript ``$-$'' again refers to the action of $T_{\text{ss}}$) and $H^k_{\van}(X_t)=0$ for $k\neq 5$.

In the terms of the vanishing-cycle sequence
$$0\to H^5(X_0)\to H^5_{\lim}(X_t)\to H^5_{\van}(X_t)\overset{\delta}{\to} H^6(X_0)\to H^6_{\lim}(X_t)\to 0$$
we thus have the following possibilities for Hodge-Deligne diagrams:
\begin{equation}\label{e76!!}
\begin{tikzpicture}[scale=0.85]
\foreach \j in {0,...,4} {\draw [thick,gray,<->] (\j*3,2) -- (\j*3,0) -- (2+\j*3,0);}
\foreach \j in {0,...,4} {\node at (\j*3+2.1,0) {\tiny $p$};}
\foreach \j in {0,...,4} {\node at (0.1+\j*3,2.1) {\tiny $q$};}
\foreach \j in {0,...,4} {\draw [gray] (\j*3-0.1,1.5) -- (\j*3+0.1,1.5);}
\foreach \j in {0,...,4} {\draw [gray] (\j*3-0.1,1) -- (\j*3+0.1,1);}
\foreach \j in {0,...,4} {\draw [gray] (\j*3-0.1,0.5) -- (\j*3+0.1,0.5);}
\foreach \j in {0,...,4} {\draw [gray] (\j*3+0.5,0.1) -- (\j*3+0.5,-0.1);}
\foreach \j in {0,...,4} {\draw [gray] (\j*3+1,0.1) -- (\j*3+1,-0.1);}
\foreach \j in {0,...,4} {\draw [gray] (\j*3+1.5,0.1) -- (\j*3+1.5,-0.1);}
\node at (1,-0.6) {$H^5(X_0)$};
\node at (4,-0.6) {$H^5_{\lim}(X_t)$};
\node at (7,-0.6) {$H^5_{\van}(X_t)$};
\node at (10,-0.6) {$H^6(X_0)$};
\node at (13,-0.6) {$H^6_{\lim}(X_t)$};
\foreach \j in {0,1} {\filldraw (1+3*\j,1) circle (2pt);}
\foreach \j in {0} {\filldraw (1.5+3.07*\j,1) circle (2pt);}
\foreach \j in {0} {\filldraw (1+2.93*\j,1.5) circle (2pt);}
\filldraw (3.93,1.5) circle (2pt);
\filldraw [red] (4.57,1) circle (2pt);
\foreach \j in {1,2,3,4} {\filldraw (1.5+3*\j,1.5) circle (2pt);}
\foreach \j in {0,1} {\node at (0.5+3*\j,0.8) {\Tiny $12+\rd-a$};}
\node at (5,1.75) {\Tiny $12+\rd-a$};
\node at (1.5+12,1.75) {\tiny $1$};
\node at (1.5+9,1.75) {\tiny $a+1$};
\draw [->] (4.4,1.4) -- (4.1,1.1);
\filldraw  (4.43,1) circle (2pt);
\filldraw [red]  (4.07,1.5) circle (2pt);
\filldraw [red] (7.5,1) circle (2pt);
\filldraw [red] (7,1.5) circle (2pt);
\node [red] at (7,1.75) {\tiny $10$};
\node [red] at (7.75,1) {\tiny $10$};
\node [red] at (4.8,1) {\tiny $10$};
\node at (4.4,0.8) {\tiny $b$};
\node at (3.75,1.5) {\tiny $b$};
\node [red] at (4.1,1.7) {\tiny $10$};
\node at (0.8,1.5) {\tiny $b$};
\node at (1.5,0.8) {\tiny $b$};
\node at (4.18,1.35) {\Tiny $N$};
\node at (7.8,1.75) {\Tiny $12+\rd$};
\end{tikzpicture}
\end{equation}
where $a=\mathrm{rk}(\delta)$ is unknown, red dots signify nontrivial $T_{\text{ss}}$-action, and
\begin{equation}\label{e76!!!}
b=h^{2,1}(X_t)-(12+\rd-a+10)=21-(22+\rd-a)=a-\rd-1.
\end{equation}
As in the case of nodes (\S\textbf{II.2.2}), more information is needed to resolve the ambiguity.

Clearly we have $1+\rd\leq a\leq 12+\rd$, hence $b\leq 11$.  We can do much better than this by generalizing the proof of Theorem \textbf{II.2.9} to arrive at the following

\begin{lem}\label{l76}
Let $I,\I$ denote the ideal sheaves of $C\disj C'$ resp.~$\cs''$ on $\PP^6$, and 
\begin{equation}\label{e76s}
\ev\colon H^0(\PP^6,\co_{\PP^6}(2))\to \frac{H^0(C\disj C',(\co_{\PP^6}/I^2)(2))}{H^0(C\sqcup C',N_{C\sqcup C'/\PP^6}(-1))}\oplus H^0(\cs'',(\co_{\PP^6}/\I)(2))
\end{equation}
the evaluation map.\footnote{The map $N_{C/\PP^6}(-1)\to (\co_{\PP^6}/I^2)|_C(2)$ (for $C$ and $C'$) is induced by $T_{\PP^6}(-1)\overset{\mathfrak{F}}{\to}\co_{\PP^6}(2)$, where $T_{\PP^6}(-1)$ is globally generated by $\{\partial_{Z_i}\}_{i=0}^6$ and $\mathfrak{F}(\partial_{Z_i}):=\partial_{Z_i}F$.  In \eqref{e76ss}, the sheaf subcomplex $\Omega_{\PP^6,I}^{\bullet}\subset \Omega_{\PP^6}^{\bullet}$ is generated by $I\Omega^1_{\PP^6}$ and $dI$, and the map $\Omega^5_{\PP^6,I}(2X_0)\to \Omega^6_{\PP^6}(3X_0)\otimes I^2(\otimes\I)$ is also induced by $F$.} Let $\TP\overset{\B}{\twoheadrightarrow}\PP^6$ be the blowup along $C\disj C'\disj \cs''$, and $\TX\overset{\b}{\twoheadrightarrow}X_0$ the strict transform, with respective exceptional divisors $\EB$ and $\eb$. Then we have identifications of vector spaces
\begin{equation}\label{e76ss}
\begin{split}
\mathrm{coker}(\ev)&\cong H^1(\PP^6,I^2\otimes \I\otimes K_{\PP^6}(3X_0))/\mathrm{im}\{H^1(\PP^6,\Omega^5_{\PP^6,I}(2X_0))\}\\
&\cong \ker(H^6(\TX)\to H^8(\TP))\cong \mathrm{im}\left(H^4(\TX)\to\tfrac{H^4(\eb)}{H^4(\EB)}\right)^{\vee}\\
&\cong \ker\left(\tfrac{H^4(\eb)}{H^4(\EB)}\twoheadrightarrow W_4H^5(X_0)\right)^{\vee}.
\end{split}
\end{equation}
\end{lem}

\begin{proof}[Sketch]
Noting that $\TX$ is smooth, Mayer-Vietoris and weak Lefschetz yield
\begin{equation}\label{e76n}
0\to H^4(\TP)\overset{\a}{\to}H^4(\TX)\overset{\beta}{\to}\tfrac{H^4(\eb)}{H^4(\EB)}\overset{\gamma}{\to}W_4 H^5(X_0)\to 0,
\end{equation}
with dual
$$0\leftarrow H^8(\TP)(1)\overset{\a^{\vee}}{\leftarrow}H^6(\TX)\overset{\beta^{\vee}}{\leftarrow}\ker\{H^4(\eb)(-1)\to H^6(\EB)\}\,;$$
and we set $H^6_0(\TX):=\mathrm{im}(\beta^{\vee})=\ker(\a^{\vee})$. By localization and \cite[Lem.~1.4]{Schoen}, 
$$H^6_0(\TX)=F^3 H^6_0(\TX)\cong F^4 H^7(\TP\stm\TX)\cong \HH^7(\TP,\Omega^{\bullet\geq 4}_{\TP}(\bullet-3)),$$
which by \cite[Lem.~4.2ff]{Bl21} is
\begin{align*}
&\cong \HH^7(\PP^6,\B_*\Omega_{\TP}^{\bullet\geq 4}(\bullet-3))\\
&\cong \HH^3(\PP^6,\Omega^4_{\PP^6}(X_0)\to \Omega^5_{\PP^6,I}(2X_0)\to \Omega^6_{\PP^6}(3X_0)\otimes I^2\otimes \I).
\end{align*}
Now $H^3(\PP^6,\Omega^4_{\PP^6}(X_0))=\{0\}$; and since
$$0\to \Omega^5_{\PP^6,I}(2X_0)\to \Omega^5_{\PP^6}(2X_0)\to N_{C\disj C'/\PP^6}(-1)\to 0$$
is exact and $N_{C\disj C'/\PP^6}(-1)\cong (\co_{\PP^6}^{\oplus 4}\oplus \co_{\PP^6}(1))|_{C\disj C'}\cong \text{``}(\co_{\PP^1}^{\oplus 4}\oplus \co_{\PP^1}(2))^{\oplus 2}\text{''}$ (by identifying $C\cong \PP^1\cong C'$), we find $H^2(\PP^6,\Omega^5_{\PP^6,I}(2X_0))= \{0\}$.  So we conclude that
$$H^6_0(\TX)\cong H^1(\PP^6,I^2\otimes \I\otimes K_{\PP^6}(3X_0))/\mathrm{im}\{H^1(\PP^6,\Omega_{\PP^6,I}^5(2X_0))\},$$
which --- by noting $\Omega^6_{\PP^6}(3X_0)\cong \co_{\PP^6}(2)$, $H^1(\PP^6,\Omega^6(3X_0))=\{0\}$, and $$\text{RHS}\eqref{e76s}\cong \frac{H^0\left(\PP^6,\Omega^6(3X_0)/(\Omega^6(3X_0)\otimes I^2\otimes \I)\right)}{H^0(\PP^6,\Omega^5_{\PP^6}(2X_0)/\Omega^5_{\PP^6,I}(2X_0))}$$ --- identifies with $\mathrm{coker}(\ev)$.
\end{proof}

\begin{thm}\label{t76}
For general masses and momenta, in dimension $D>4$ \textup{[}resp.~$D=4$\textup{]}, the second Symanzik hypersurface for the double box has Hodge numbers $h^{2,2}(H^5(X_0))=9$ \textup{[}resp.~$10$\textup{]} and $h^{3,2}(H^5(X_0))=h^{2,3}(H^5(X_0))=2$ \textup{[}resp.~$1$\textup{]}. The LMHS $H^5_{\lim}$ for the general $1$-parameter cubic smoothing has $\mathrm{rk}(N)=h^{3,3}=h^{2,2}=9$ \textup{[}resp.~$10$\textup{]} and $h^{2,2}=h^{3,3}=12$ \textup{[}resp.~$11$\textup{]} while $T_{\textup{ss}}$ acts by $-1$ on a subspace of $\gr^W_{5}$ of rank $20$.  The degeneration from $D>4$ to $D=4$, in which $X_0$ acquires $2$ nodes, has monodromy logarithm of rank $1$.
\end{thm}

\begin{proof}
By \eqref{e76!!}-\eqref{e76!!!}, it will suffice to show that $a=3$ ($D>4$) resp.~$4$ ($D=4$). Note that by \eqref{e76!!}, $\mathrm{rk}(W_4H^5(X_0))=12+\rd-a$.

First, we claim that $a=\mathrm{rk}(\mathrm{coker}(\ev))$. Each of the quadric 3-fold bundles in $\eb$ (over $C$ and $C'$) has 6 corank-1 singular fibers, hence contributes $8$ to $h^4(\eb)$ and $6$ to $\mathrm{rk}(\tfrac{H^4(\eb)}{H^4(\EB)})$ (while each quadric 4-fold arising from a node contributes $2$ resp.~$1$). By the Lemma, $\mathrm{rk}(\mathrm{coker}(\ev))$ is thus
$$\mathrm{rk}(\tfrac{H^4(\eb)}{H^4(\EB)})-\mathrm{rk}(W_4H^5(X_0))=2\cdot 6+\rd\cdot 1-(12+\rd-a)=a,$$
as desired.

To compute $\ev$ when $D>4$, write $S^2$ for homogeneous polynomials in $Z_0,\ldots,Z_6$ of degree $2$, and $(0123)^2$ [resp.~$(3456)^2$; $J_F$] for the subspace given by polynomials in $Z_0,\ldots,Z_3$ [resp.~polynomials in $Z_3,\ldots,Z_6$; linear combinations of $\partial_{Z_0}F,\ldots,\partial_{Z_6}F$]. Recalling $C=\{Z_0=\cdots=Z_3=Q=0\}$, the only degree-$2$ polynomials vanishing to order $2$ on $C$ are those in $(0123)^2$, making this the kernel of the map from $S^2$ to $H^0(C,(\co_{\PP^6}/I^2)(2))$.  As the latter space and $S^2/(3456)^2$ both have dimension $18$,\footnote{$\dim(S)=\binom{8}{2}=28$, $\dim((3456)^2)=\binom{5}{2}=10$, and an easy computation using $0\to N^{\vee}_{C/\PP^6}(2)\to (\co_{\PP^6/I^2})|_C(2)\to\co_C(2)\to 0$ and $N^{\vee}_{C/\PP^6}(2)\cong \co_C(1)^{\oplus 4}\oplus \co_C$.} they are isomorphic. Similarly, one finds that $\CC\langle \{\partial_{Z_i}\}_{i=0}^6\rangle=H^0(\PP^6,T_{\PP^6}(-1))\to H^0(C,N_{C/\PP^6}(-1))$ is an isomorphism by a sheaf cohomology computation, whence $\ev$ is just the natural map
\begin{equation}\label{e76s?}
\ev\colon S^2\to \frac{S^2}{(0123)^2+J_F}\oplus \frac{S^2}{(3456)^2+J_F}.
\end{equation}

Note that $(0123)^2\cap (3456)^2=\CC\langle Z_3^2\rangle$. We claim that 
\begin{equation}\label{e76s!}
((0123)^2+J_F)\cap ((3456)^2+J_F)=J_F\oplus \CC\langle Z_3^2\rangle \oplus \CC\langle U\rangle,
\end{equation}
making $\dim(\ker(\ev))=9$. To see this, note that the ``cross-terms'' of type $(012)(456)$ in the $\partial_{Z_i}F$ are (in order)
$$\{Z_{456}\partial_{Z_i}Q'\}_{i=0}^2,\;\;P+\ms_3^2Z_3\partial_{Z_3}U,\;\;\{Z_{012}\partial_{Z_i}Q\}_{i=4}^6.$$
Denoting by $c_0,c_1,c_2,c_4,c_5,c_6$ the unique constants that make $$\textstyle\sum_{i=0}^2c_i\partial_{Z_i}Q'=Z_{012}\;\;\;\text{and}\;\;\;\sum_{i=4}^6c_i\partial_{Z_i}Q=Z_{456},$$ we set $\partial'F:=\sum_{i=0}^2c_i\partial_{Z_i}F$, $\partial F=\sum_{i=4}^6c_i\partial_{Z_i}F$. Then $\delta F:=\partial'F-\partial F$ is (up to scale) the unique element of $J_F$ with no cross-terms (by genericity of $P$), and we may write uniquely mod $\CC\langle Z_3^2\rangle$ $\delta F=:g-g'$ with $g\in (3456)^2$ and $g'\in (0123)^2$.  One checks that $g=\partial'F-U$ and $g'=\partial F-U$ (whence $U\in \text{LHS}\eqref{e76s!}$). Now given $G_0\in \text{LHS}\eqref{e76s!}$, we have $G'+H'=G_0=G+H$ with $G'\in (0123)^2$, $G\in (3456)^2$, and $H,H'\in J_F$.  Then $H-H'=\sum_{i=0}^6 \tilde{c}_i\partial_{Z_i}F=G'=G$ has no cross-terms and must be a multiple of $\delta F$, whence $G=\mu g+\lambda Z_3^2$, and $G_0\in \text{RHS}\eqref{e76s!}$.

From \eqref{e76s!} we can now immediately read off 
\begin{align*}
\mathrm{rk}(\mathrm{coker}(\ev))&=\dim (\text{RHS}\eqref{e76s?})-(\dim(S^2)-9)\\
&=2(28-10-7)-(28-9)=3,
\end{align*}
proving the Theorem for $D>4$. For $D=4$, $\ev$ is ``enriched'' by the evaluation at the two nodes, becoming
$$\ev\colon S^2\to \frac{S^2}{(0123)^2+J_F}\oplus \frac{S^2}{(3456)^2+J_F} \oplus \frac{S^2}{\I}$$
(with $S^2/\I\cong \CC^2$). Since all the $\partial_{Z_i}F$, as well as $U$,  vanish on $\cs''$, they map to $0$ in $S^2/\I$;\footnote{The presence of $U$ in this kernel implies that the map $H^1(\PP^1,\Omega_{\PP^6,I}(2X_0))\to H^1(\PP^6,\Omega^6_{\PP^6}(3X_0)\otimes I^2\otimes \I)$, claimed to be injective in \cite[Lem.~5.3]{Bl21}, in fact has a rank 1 kernel.} on the other hand, $Z_3^2$ does not map to $0$ in $S^2/\I$. So $\mathrm{rk}(\ker(\ev))$ has increased by $1$ and the rank of the codomain by $2$, making $\mathrm{rk}(\mathrm{coker}(\ev))=3+2-1=4$.
\end{proof}

We conclude with some remarks on the extensions of MHS in \eqref{e76!!}.  As usual, $H^5_{\lim}$ and $H^5_{\van}$ are the direct sums of their $T_{\text{ss}}$-invariants (black) and other $T_{\text{ss}}$-eigenspaces (red). The remaining biextension on the ``black'' part of $H_{\lim}^5$ includes the extension class of 
\begin{equation}\label{e76ext}
0\to W_4 H^5(X_0)\to H^5(X_0)\to \gr^W_5 H^5(X_0)\to 0,
\end{equation}
which is not hard to describe heuristically. Indeed, by \eqref{e76n} we have $(W_4H^5(X_0))^{\vee}\cong \ker\{H_4(\eb)\to H_4(\TX)\oplus H^4(\EB)\}$, which is represented by certain algebraic 2-cycles $Z\in \mathrm{CH}^2(\eb)$ which in particular bound in $\TX$: writing $\imath\colon \eb\hookrightarrow \TX$ for the inclusion, we have $\imath_*(Z)=\partial \Gamma$ for some 5-chain $\Gamma$ on $\TX$. Noting that $\gr^W_5 H^5(X_0)\cong H^5(X_0)$, \eqref{e76ext} is then computed by the Abel-Jacobi classes
\begin{equation}\label{e76aj}
\mathrm{AJ}_{\TX}(Z)=\int_{\Gamma}(\cdot)\in \frac{\{F^3 H^5(\TX,\CC)\}^{\vee}}{H_5(\TX,\ZZ)}=J^3(\TX)
\end{equation}
of such $Z$.

Now assume $D>4$. By \cite{DHV23}, $\TX$ has a birational model $\hat{X}_0$ with a map $\Pi\colon \hat{X}_0 \to\PP^1$, whose fibers are generically smooth quadric 4-folds, with 6 corank-1 fibers over points $\Sigma\subset \PP^1$. (As a rational map on $X_0$, $\Pi$ is just given by $U(\UZ)/Z_3^2$.) Let $\mathscr{C}\overset{\pi}{\to} \PP^1$ denote the (genus $2$) double-cover branched over $\Sigma$.  There exists a $\PP^2$-bundle $\mathscr{Y}\overset{\rho}{\to}\mathscr{C}$ and an embedding $I\colon \mathscr{Y}\hookrightarrow \hat{X}_0$ such that $\Pi\circ I=\pi\circ \rho$; let $Y\overset{\tilde{\rho}}{\to}\mathscr{C}$ be the birational model with an embedding $\tilde{I}\colon Y\hookrightarrow \TX$.  Then $\tilde{I}_*\circ \tilde{\rho}^*\colon H^1(\mathscr{C})(-2)\overset{\cong}{\to}H^5(\TX)$ recovers the isomorphism of [op.~cit.]. Write $\{\omega_i\}_{i=1}^2\subset \Omega^1(\mathscr{C})$ for a basis and $\Omega_i:=\tilde{I}_*(\tilde{\rho}^*\omega_i)$ for 5-currents spanning $F^3H^5(\TX)$.

Finally, for $Z\in \mathrm{CH}^2(\eb)$ as above, we set 
\begin{equation}\label{e76cyc}
Z_{\mathscr{C}}:=\tilde{\rho}_*(Y\cdot\imath_*Z) \in CH^1_{\text{hom}}(\mathscr{C}).
\end{equation}
Then we have
$$\mathrm{AJ}_{\TX}(Z)(\Omega_i)=\int_{\Gamma}\tilde{I}_*\tilde{\rho}^*\omega_i=\int_{\Gamma\cap Y}\tilde{\rho}^*\omega_i=\int_{\rho_*(\Gamma\cap Y)}\omega_i=\mathrm{AJ}_{\mathscr{C}}(Z_{\mathscr{C}})(\omega_i)$$
since $\partial\tilde{\rho}(\Gamma\cap Y)=\tilde{\rho}(\partial\Gamma\cap Y)=Z_{\mathscr{C}}$, and so \eqref{e76aj} (hence \eqref{e76ext}) is simply computed by the Abel-Jacobi images of the cycles \eqref{e76cyc} on $\mathscr{C}$.

\bibliography{csseq2}

\providecommand{\bysame}{\leavevmode\hbox to3em{\hrulefill}\thinspace}
\providecommand{\MR}{\relax\ifhmode\unskip\space\fi MR }
\providecommand{\MRhref}[2]{%
  \href{http://www.ams.org/mathscinet-getitem?mr=#1}{#2}
}
\providecommand{\href}[2]{#2}
\begin{thebibliography}{MOPW23}

\bibitem[Blo21]{Bl21}
S.~Bloch, \emph{Double-box motive}, SIGMA \textbf{17} (2021), 048, 12 pages,
  arXiv:2105.06132.

\bibitem[CDKP22]{CDKP}
B.~Castor, H.~Deng, M.~Kerr, and G.~Pearlstein, \emph{Remarks on eigenspectra
  of isolated singularities}, arXiv:2211.04648, 2022.

\bibitem[Cle77]{Clemens}
C.~H. Clemens, \emph{Degeneration of {K}\"ahler manifolds}, Duke Math. J.
  \textbf{44} (1977), no.~2, 215--290.

\bibitem[DHV23]{DHV23}
C.~Doran, A.~Harder, and P.~Vanhove, \emph{Motivic geometry of two-loop feynman
  integrals}, arXiv:2302.14840, 2023.

\bibitem[EFM21]{EFM}
D.~Eriksson, G.~Freixas, and C.~Mourougane, \emph{B{COV} invariants of
  {C}alabi-{Y}au manifolds and degenerations of {H}odge structures}, Duke Math.
  J. \textbf{170} (2021), no.~3, 379--454.

\bibitem[FL22]{FL-DuBois}
R.~Friedman and R.~Laza, \emph{Higher {D}u {B}ois and higher rational
  singularities}, with an Appendix by M. Saito, arXiv:2205.04729, 2022.

\bibitem[Gri21]{Griffiths}
P.~Griffiths, \emph{Hodge theory and moduli}, Geometry at the
  frontier---symmetries and moduli spaces of algebraic varieties, Contemp.
  Math., vol. 766, Amer. Math. Soc., [Providence], RI, [2021] \copyright 2021,
  pp.~163--200.

\bibitem[JKSY19]{JKSY2}
S.-J. Jung, I.-K. Kim, M.~Saito, and Y.~Yoon, \emph{Spectrum of non-degenerate
  functions with simplicial {N}ewton polyhedra}, 1911.09465, 2019.

\bibitem[JKSY22]{JKSY-DuBois}
\bysame, \emph{Higher {D}u {B}ois singularities of hypersurfaces}, Proc. Lond.
  Math. Soc. (3) \textbf{125} (2022), no.~3, 543--567.

\bibitem[KK10]{KK}
J.~Koll{\'a}r and S.~J. Kov{\'a}cs, \emph{Log canonical singularities are {D}u
  {B}ois}, J. Amer. Math. Soc. \textbf{23} (2010), no.~3, 791--813.

\bibitem[KL21]{KL1}
M.~Kerr and R.~Laza, \emph{Hodge theory of degenerations, ({I}): consequences
  of the decomposition theorem, with an appendix by {M.} {S}aito}, Selecta
  Math. (N.S.) (2021), no.~4, Paper No. 71, 48.

\bibitem[KL23]{KL}
M.~Kerr and R.~Laza, \emph{Hodge theory of degenerations, {(II)}: Methods for
  computing the vanishing cohomology}, to appear, 2023.

\bibitem[KLS22]{KLS}
M.~Kerr, R.~Laza, and M.~Saito, \emph{Deformation of rational singularities and
  {H}odge structure}, Algebr. Geom. \textbf{9} (2022), no.~4, 476--501.

\bibitem[Kol23]{Kollar-book}
J.~Koll\'{a}r, \emph{Families of varieties of general type}, Cambridge Tracts
  in Mathematics, vol. 231, Cambridge University Press, 2023.

\bibitem[KSB88]{ksb}
J.~Koll{\'a}r and N.~I. Shepherd-Barron, \emph{Threefolds and deformations of
  surface singularities}, Invent. Math. \textbf{91} (1988), no.~2, 299--338.

\bibitem[LO18]{log2}
R.~Laza and K.~G. O'Grady, \emph{G{IT} versus {B}aily-{B}orel compactification
  for quartic {$K3$} surfaces}, Geometry of moduli, Abel Symp., vol.~14,
  Springer, Cham, 2018, pp.~217--283.

\bibitem[LR12]{LR}
W.~Liu and S.~Rollenske, \emph{Two-dimensional semi-log-canonical
  hypersurfaces}, Matematiche (Catania) \textbf{67} (2012), no.~2, 185--202.

\bibitem[MOPW23]{MOPW}
M.~Musta\c{t}\u{a}, S.~Olano, M.~Popa, and J.~Witaszek, \emph{The {D}u {B}ois
  complex of a hypersurface and the minimal exponent}, Duke Math. J.
  \textbf{172} (2023), no.~7, 1411--1436.

\bibitem[MP18]{MP18i}
M.~Musta\c{t}\u{a} and M.~Popa, \emph{Restriction, subadditivity, and
  semicontinuity theorems for {H}odge ideals}, Int. Math. Res. Not. IMRN
  (2018), no.~11, 3587--3605.

\bibitem[MP20]{MP18ii}
\bysame, \emph{Hodge ideals for {$\mathbb{Q}$}-divisors, {$V$}-filtration, and
  minimal exponent}, Forum Math. Sigma \textbf{8} (2020), e19.

\bibitem[MP22]{MP-rat}
\bysame, \emph{On $k$-rational and $k$-{D}u {B}ois local complete
  intersections}, arXiv:2207.08743, 2022.

\bibitem[MY23]{MY}
L.~Maxim and R.~Yang, \emph{Higher {D}u {B}ois and higher rational
  singularities of hypersurfaces}, arXiv:2301.09084, 2023.

\bibitem[N\'91]{Nemethi91}
A.~N\'{e}methi, \emph{Generalized local and global {S}ebastiani-{T}hom type
  theorems}, Compositio Math. \textbf{80} (1991), no.~1, 1--14.

\bibitem[Sai90]{Sai90b}
M.~Saito, \emph{Mixed {H}odge modules}, Publ. Res. Inst. Math. Sci. \textbf{26}
  (1990), 221--333.

\bibitem[Sai91]{Sa-steen}
\bysame, \emph{On {S}teenbrink's conjecture}, Math. Ann. \textbf{289} (1991),
  no.~4, 703--716.

\bibitem[Sai09]{Sa5}
\bysame, \emph{On the {H}odge filtration of {H}odge modules}, Mosc. Math. J.
  \textbf{9} (2009), no.~1, 161--191, back matter.

\bibitem[Sai16]{Sai16}
\bysame, \emph{Hodge ideals and microlocal {$V$}-filtration}, arXiv:1612.08667,
  2016.

\bibitem[Sai17]{Sai17}
\bysame, \emph{Roots of {B}ernstein-{S}ato polynomials for projective
  hypersurfaces with general hyperplane sections having weighted homogeneous
  isolated singularities}, arXiv:1703.05741v4, 2017.

\bibitem[Sai20]{Sai20}
\bysame, \emph{Descent of nearby cycle formula for newton non-degenerate
  functions}, arXiv:2004.12367, 2020.

\bibitem[Sak74]{Sakamoto74}
K.~Sakamoto, \emph{The {S}eifert matrices of {M}ilnor fiberings defined by
  holomorphic functions}, J. Math. Soc. Japan \textbf{26} (1974), 714--721.

\bibitem[Sch85]{Schoen}
C.~Schoen, \emph{Algebraic cycles on certain desingularized nodal
  hypersurfaces}, Math. Ann. \textbf{270} (1985), no.~1, 17--27.

\bibitem[Sie90]{Si90}
D.~Siersma, \emph{The monodromy of a series of hypersurface singularities},
  Comment. Math. Helv. \textbf{65} (1990), no.~2, 181--197.

\bibitem[Ste89]{St89}
J.~H.~M. Steenbrink, \emph{The spectrum of hypersurface singularities},
  Ast\'{e}risque (1989), no.~179-180, 11, 163--184, Actes du Colloque de
  Th\'{e}orie de Hodge (Luminy, 1987).

\bibitem[Ste76]{St2}
\bysame, \emph{Limits of {H}odge structures}, Invent. Math. \textbf{31}
  (1975/76), no.~3, 229--257.

\end{thebibliography}
 
\end{document}